\documentclass[12pt]{amsart}
\usepackage{amsfonts, amsbsy, amsmath, amssymb}

\usepackage[T1]{fontenc}
\usepackage[utf8]{inputenc}
\usepackage{lmodern}
\usepackage{caption}
\usepackage{amsmath, amsthm, amssymb,amscd, mathrsfs, amsfonts, mathtools,tikz-cd}
\usepackage{relsize}
\usepackage{euler}
\usepackage{pbox}
\usepackage{times}
\usepackage[all]{xy}
\usepackage{todonotes}
\usepackage{xcolor}
\usepackage[bookmarks=true,
bookmarksnumbered=true, breaklinks=true,
pdfstartview=FitH, hyperfigures=false,
plainpages=false, naturalnames=true,
colorlinks=true,pagebackref=true,
pdfpagelabels]{hyperref}

\usepackage{tikz}
\usepackage{pgfplots}

\pgfplotsset{compat=1.10}
\usepgfplotslibrary{fillbetween}
\usetikzlibrary{patterns}
\usetikzlibrary{matrix}

\usepackage[lite,alphabetic]{amsrefs}

\renewcommand{\PrintDOI}[1]{\href{http://dx.doi.org/\detokenize{#1}}{doi: \detokenize{#1}}%
	\IfEmptyBibField{pages}{, (to appear in print)}{}}

\def\commutatif{\ar@{}[rd]|{\circlearrowleft}}

\newtheorem{thm}{Theorem}[section]
\newtheorem{prop}[thm]{Proposition}

\newtheorem{cor}[thm]{Corollary}
\newtheorem{conj}[thm]{Conjecture}

\theoremstyle{definition}
\newtheorem{defn}[thm]{Definition}

\newtheorem*{qst}{Question}

\theoremstyle{remark}
\newtheorem{rmk}[thm]{Remark}

\newtheorem{ex}[thm]{Example}

\allowdisplaybreaks

\newcommand\Hom{\operatorname{Hom}}

\usepackage{bbm}

\def\End{\operatorname{End}}

\hypersetup{
	colorlinks = true,
	urlcolor = blue,
	linkcolor = blue,
	citecolor = red,
	pdfpagemode = UseNone
}

\title{The universe inside Hall algebras of coherent sheaves on  toric resolutions}

\begin{document}
		\begin{center}
		\maketitle 	
		\large
		Boris Tsvelikhovskiy \\

		Department of Mathematics,\\
		University of Pittsburgh, 15260, Pittsburgh,~USA\\
		E-mail: bdt18@pitt.edu\hspace{3in}\\
	\end{center}
	
	\textit{Mathematics Subject Classification (2020)}: 14F08, 14J17, 17B35

	\begin{abstract} Let $\mathfrak{g}\neq \mathfrak{so}_8$ be  a simple Lie algebra of type $A,D,E$  with $\widehat{\mathfrak{g}}$ the corresponding  affine Kac-Moody algebra and $\mathfrak{n}_-\subset \widehat{\mathfrak{g}}$  a nilpotent subalgebra. Given $\mathfrak{n}_-$ as above, we provide an infinite collection of cyclic finite abelian subgroups of $SL_3(\mathbb{C})$ with the following properties. Let $G$ be any group in the collection, $Y=G\operatorname{-}\mbox{Hilb}(\mathbb{C}^3)$ and $\Psi: D^b_G(Coh(\mathbb{C}^3))\rightarrow D^b(Coh(Y))$ the derived equivalence of Bridgeland, King and Reid. We present an (explicitly described) subset of objects in $Coh_G(\mathbb{C}^3)$, s.t. the Hall algebra generated by their images under $\Psi$ is isomorphic to $U(\mathfrak{n}_-)$. In case the field $\Bbbk$ (in place of $\mathbb{C}$) is finite and $\mbox{char}(\Bbbk)$ is coprime with the order of $G$, we conjecture the isomorphisms of the corresponding 'counting' Ringel-Hall algebras and the specializations of quantized universal enveloping algebras $U_v(\mathfrak{n}_-)$ at $v=\sqrt{|\Bbbk|}$.    
	\end{abstract}
\section{Introduction}
We begin with describing different versions of McKay correspondence. Let $G\subset GL_n(\mathbb{C})$ be a finite subgroup and consider the categorical quotient $X=\mathbb{C}^n/G:=Spec(\mathbb{C}[x_1,x_2,\hdots,x_n]^G)$. In addition assume the following properties:
\begin{itemize}
	\item $X$ has an isolated singularity at $0$;
	\item there exists a projective resolution $\rho: Y \rightarrow X$. 
\end{itemize}

One way to understand McKay correspondence is via the bijection $$\left\{
\begin{aligned}
	&\mbox{nontrivial~irreducible}\\
	&\mbox{representations~of  }G
\end{aligned}\right\}\overset{1:1}{\longleftrightarrow}
\left\{
\begin{aligned}
	&	\mbox{irreducible ~components}\\
	&	\mbox{of ~the ~central~ fiber } \rho^{-1}(0)
\end{aligned}\right\}.$$

 Let $Coh_G(\mathbb{C}^n)$ be the category of $G$-equivariant coherent sheaves on $\mathbb{C}^n$ and $Coh(Y)$
be the category of coherent sheaves on $Y$. In modern language the McKay correspondence is usually understood as an equivalence of triangulated categories $D_G^b(Coh(\mathbb{C}^n))$ and $D^b(Coh(Y))$. Such an equivalence is known to hold in the following cases: 
\begin{itemize}
	\item $G\subset SL_2{(\mathbb{C})}$, any $G$ (\cite{KV});
	\item $G\subset SL_3{(\mathbb{C})}$, any $G$, $Y=G\operatorname{-}\mbox{Hilb}(\mathbb{C}^3)$ (\cite{BKR});
	\item $G\subset SL_3{(\mathbb{C})}$, any abelian $G$ and any crepant resolution $(Y,\rho)$ (\cite{CI});
	\item  $G\subset SP_{2n}(\mathbb{C})$, any $G$ and crepant symplectic resolution  $(Y,\rho)$ (\cite{BK});
	\item $G\subset SL_n{(\mathbb{C})}$, any abelian $G$,  any crepant symplectic resolution $(Y,\rho)$ (\cite{KawT}).
\end{itemize}

Any finite-dimensional representation $V$ of $G$ gives rise to two equivariant sheaves on
$\mathbb{C}^n$: 
\begin{itemize}
	\item the skyscraper sheaf $V^!=V\otimes_{\mathbb{C}} \mathcal{O}_{0}$, whose fiber at $0$ is $V$ and all the other fibers vanish;
	\item  the locally free sheaf $\widetilde{V}=V\otimes_{\mathbb{C}} \mathcal{O}_{\mathbb{C}^n}$.
\end{itemize}

 Suppose the map $\Psi: D^b_G(\mathbb{C}^n)\overset{\sim}{\rightarrow }D^b(Y)$ is an exact equivalence of triangulated categories.

A natural question emerges.  
\begin{qst}
	What are the images of $\widetilde{\chi}$ and $\chi^!$ (for nontrivial irreducible representations $\chi$ of $G$) under the equivalence $\Psi$?
\end{qst}
The former, $\Psi(\widetilde{\chi})$, is a vector bundle of dimension dim$(\chi)$ and is called a tautological or GSp-V sheaf (after Gonzales-Sprinberg and Verdier). Relatively little is known about $\Psi(\chi^!)$.

The following results are due to Kapranov, Vasserot (see \cite{KV}) 
and Cautis, Craw, Logvinenko (see \cite{CCL}).

	\begin{enumerate}
		\item Let $G\subset SL_2(\mathbb{C})$ be a finite subgroup and $\chi \in \mbox{Irr}(G)\setminus triv$. Then $\Psi(\chi^!)\simeq \mathcal{O}_{\mathbb{P}^1_{\chi}}(-1)[1]$.
		\item Let $G\subset SL_3(\mathbb{C})$ be a finite abelian subgroup. Then for any $\rho \in \mbox{Irr}(G)\setminus triv$, the object $\Psi(\chi^!)\in D^b(Coh(Y))$ is pure (here $Y=G\operatorname{-}Hilb(\mathbb{C}^3)$ and an object is called \textbf{pure} provided all cohomology, except in a single degree, vanish). 
\end{enumerate}

Let $\mathcal{C}$ be a $\mathbb{C}$-linear finitary abelian category (the former means that for any two objects $A,B \in \mathcal{C}$ and $i \in \mathbb{Z}_{\geq 0}$ one has $\dim(\mbox{Ext}^i(A,B))<\infty$). There is a way to associate an algebra $\mathcal{H}(\mathcal{C})$ to $\mathcal{C}$, which is called the Hall algebra of $\mathcal{C}$ (see Section $4$ and references therein). In Section $3$ of \cite{KV} it was observed that in case $\mathcal{C}_1$  and  $\mathcal{C}_2$ are  $\mathbb{C}$-linear finitary abelian categories, there is a derived equivalence $\Psi: D^b(\mathcal{C}_1)\rightarrow D^b(\mathcal{C}_2)$ and a collection of objects $\{a_1,\hdots,a_n\}$ in $\mathcal{C}_1$, s.t. $\Psi(a_i)$ are all pure and concentrated in the same degree, then the Hall algebra generated by the objects $\{a_1,\hdots,a_n\}$ is isomorphic to the Hall algebra generated by their images $\{\Psi(a_1),\hdots,\Psi(a_n)\}$.

Next suppose $Q$ is a quiver with no oriented cycles and let $Rep(Q)$ be the category of representations of $Q$. Let $\mathcal{P}_Q$ be the path algebra of $Q$. It is well known that the categories $Rep(Q)$ and $\mathcal{P}_Q\operatorname{-mod}$ are equivalent. The assumptions on $Q$ imply that the category $\mbox{Rep}(Q)$ is hereditary, i.e. $\mbox{Ext}^i(A,B)=0$ for any $i\geq 2$. Furthermore, both $\dim(\mbox{Hom}(A,B))$ and $\dim(\mbox{Ext}^1(A,B))<\infty$, while the presence of equivalence of categories $\mbox{Rep}(Q)\simeq \mathcal{P}_Q\operatorname{-mod}$ with the latter category being abelian guarantees that $\mbox{Rep}(Q)$ is an abelian category as well.   There is a natural way to associate a Kac-Moody Lie algebra $\mathfrak{g}_Q$ to $Q$. Namely, the Cartan matrix for $\mathfrak{g}_Q$ is $C=2\cdot I-A_Q-A_Q^T$, where $A_Q$ is the adjacency matrix of $Q$.

 There is an embedding $U(\mathfrak{n}_-)\rightarrow\mathcal{H}(\mbox{Rep}(Q))$, where $\mathcal{H}(\mbox{Rep}(Q))$ is the Hall algebra for the category of quiver representations. The image is isomorphic to the composition subalgebra of $\mathcal{H}(Rep(Q))$ generated by characteristic functions of simple representations, see Example $4.25$ in \cite{J}. 

Let $G\subset GL_n(\mathbb{C})$ be a finite group. The McKay quiver $Q(G,\mathbb{C}^n)$ is the quiver whose vertices are in bijection with irreducible representations of $G$ joined by $\dim(\Hom_G(\chi_k\otimes \mathbb{C}^n,\chi_\ell))$ edges (possibly none) from vertex $k$ to vertex $\ell$ (here $\chi_k$ and $\chi_{\ell}$ are irreducible representations of $G$). The category of representations of McKay quiver $Q(G,\mathbb{C}^n)$ is equivalent to the category  $\mathcal{P}_{Q(G,\mathbb{C}^n)}/\mathcal{I}\operatorname{-mod}$, where $\mathcal{I}\subset\mathcal{P}_{Q(G,\mathbb{C}^n)}$ is an ideal. For instance, this ideal admits a nice description, when $G$ is abelian (see Section $3.2$). The category of McKay quiver representations will be denoted by $\mbox{Rep}(Q(G,\mathbb{C}^n),\mathcal{R})$ (here $\mathcal{R}$ stands for 'relations', see Section $3.2$ for details). There is an equivalence of abelian categories $\mbox{Rep}(Q(G,\mathbb{C}^n),\mathcal{R})\simeq Coh_G(\mathbb{C}^n)$.

 Suppose $G\subset GL_n(\mathbb{C})$ satisfies the following assumptions:
\begin{enumerate}
	\item  the McKay quiver $Q(G,\mathbb{C}^n)$ contains a subquiver $Q'$ (without oriented cycles) with $\mathcal{I}\cap \mathcal{P}_{Q'}=0$;
	\item there is a derived equivalence $\Psi:  D_G^b(Coh(\mathbb{C}^n))\rightarrow D^b(Coh(Y))$;
	\item  $\Psi$ sends the skyscraper sheaves $\chi^!\in Coh_G(\mathbb{C}^n)$, corresponding to the simple representations in $\mbox{Rep}(Q(G,\mathbb{C}^n),R)$ supported at the vertices of $Q'$, to  pure sheaves concentrated in the same degree.
\end{enumerate}  

Let $\mathcal{H}\langle \{\Psi(\chi_i^!)\}_{i\in Q'_0}\rangle$ be the Hall algebra generated by the images of sheaves corresponding to simple representations of $Q'$ under $\Psi$ and $\mathfrak{n}_-\subset\mathfrak{g}_{Q'}$ stand for the corresponding nilpotent subalgebra of $\mathfrak{g}_{Q'}$.  It follows from the discussion above that one has an isomorphism of algebras (see also diagram \ref{diagram}):

$$\Theta:U(\mathfrak{n}_-)\rightarrow \mathcal{H}\langle \{\Psi(\chi_i^!)\}_{i\in Q'_0}\rangle.$$

 In this paper we present an infinite collection of cyclic finite abelian subgroups of $SL_3(\mathbb{C})$ satisfying conditions $(1)-(3)$ above with $Q'$  any simply laced Dynkin diagram of affine type except $\widetilde{D}_4$ or $\widetilde{A}_0$, hence, produce isomorphisms $U(\mathfrak{n}_-)\overset{\simeq}{\rightarrow}\mathcal{H}\langle \{\Psi(\chi_i^!)\}_{i\in Q'_0}\rangle$ for  $\mathfrak{n}_-\subset \widehat{\mathfrak{g}}$ with $\mathfrak{g}\neq \mathfrak{so}_8$  a simple Lie algebra of type $A,D,E$  and $\widehat{\mathfrak{g}}$ the corresponding  affine Kac-Moody algebra. Let $\varepsilon=e^{2\pi i/r}$ be the primitive root of unity. The indicated family consists of cyclic abelian groups $\varphi:\mathbb{Z}/r\mathbb{Z}\hookrightarrow SL_3(\mathbb{C})$ with $\varphi(1)=\mbox{diag}(\varepsilon, \varepsilon^k, \varepsilon^s)$, where $r=1+k+s$ and 
 \begin{itemize}
 	\item $s\equiv 0 ~(\mbox{mod } k)$,
 	\item $s\equiv 0 ~(\mbox{mod } k+1)$.
 \end{itemize}

Interestingly, as $s$ goes to infinity, the proportion of nontrivial characters $\{\chi~|~H^0(\Psi(\chi^!))\neq 0\}$ tends to $\dfrac{k-1}{k+1}$ (see Corollary \ref{LimRatio}). In particular, as both $s$ and $k$ tend to infinity, for a uniformly randomly chosen character $\chi$, one has that $\Psi(\chi^!)$ is concentrated in degree $0$ with probability $1$. 

The exposition in the paper is organized as follows. In Sections $2-4$ we recall the generalities on McKay correspondence, quiver representations and Hall algebras, respectively. Each section has references for a more detailed overview of the corresponding topic. While these sections contain essentially no new results, the example presented in Section $2$ is important for a better understanding of the construction. Finally, Section $5$ introduces the families of finite cyclic subgroups in $SL_3(\mathbb{C})$ and establishes the isomorphisms of algebras announced above  (Theorems \ref{CohMin1} and \ref{MainThm}). 

	\textbf{Acknowledgement:} I am grateful to Roman Bezrukavnikov and Timothy Logvinenko for enlightening discussions. I would like to thank Timothy Logvinenko for drawing my attention to derived Reid's recipe and explaining how it works. I am indebted to Michael Finkelberg and Olivier Schiffmann for valuable suggestions.  
\section{McKay correspondence}
We start with a quick chronological overview of the subject. In \cite{McKay} John McKay has observed that  for
a finite subgroup $G\subset SL_2(\mathbb{C})$ there is a bijection

$$\left\{
\begin{aligned}
&\mbox{nontrivial~irreducible}\\
&\mbox{representations~of  }G
\end{aligned}\right\}\overset{1:1}{\longleftrightarrow}
\left\{
\begin{aligned}
&	\mbox{irreducible ~components}\\
&	\mbox{of ~the ~central~ fiber } \rho^{-1}(0)
\end{aligned}\right\},$$
where $\rho:Y\rightarrow \mathbb{C}^2/G$ is the minimal resolution of singularities. Notice that $\mathbb{C}[G]$ (the ring of representations of $G$) is naturally isomorphic to $K^G(\mathbb{C}^2)$, the Grothendieck group of
$G$-equivariant coherent sheaves on $\mathbb{C}^2$. Following this observation, in \cite{GSV}, the McKay correspondence was realized
geometrically  as an isomorphism of Grothendieck groups $K^G(\mathbb{C}^2)\rightarrow K(Y)$. Next in \cite{KV} this isomorphism was lifted to an equivalence of triangulated categories of coherent sheaves: $$D^G(\mathbb{C}^2)\rightarrow D(Y).$$ In particular, under this equivalence, $\chi^!:=\chi\otimes\mathcal{O}_0$,
the skyscraper sheaf at $0$ associated to a nontrivial irreducible $G$-representation $\chi$, is mapped to the structure sheaf
of the corresponding exceptional divisor (twisted by $\mathcal{O}(-1)$). Then Bridgeland, King and Reid constructed the equivalence $D^G(\mathbb{C}^3)\rightarrow D(Y)$ for any finite subgroup $G\subset SL_3(\mathbb{C})$ and $Y = G\operatorname{-}\mbox{Hilb}(\mathbb{C}^3)$ (see \cite{BKR}). They showed that $G\operatorname{-}\mbox{Hilb}(\mathbb{C}^3)$  is a crepant resolution of $\mathbb{C}^3$.  It was established that the images of $\chi^!$s are concentrated in a single degree in case $G$ is abelian and $\mathbb{C}^3/G$ has a single isolated singularity (see \cite{CL}). The result was then extended to any finite abelian subgroup $G$
of $SL_3(\mathbb{C})$ in \cite{CCL}. We briefly recall the setup.

\subsection{$G\operatorname{-}\mbox{Hilb}(\mathbb{C})^3$ as a toric variety}  We  refer the reader to Section $2$ of \cite{C1} and \cite{CR} for a more comprehensive and detailed exposition.

Let $G\subset SL_3(\mathbb{C})$ be a finite abelian subgroup of order $r = |G|$, and $\varepsilon=e^{2\pi i/r}$ a primitive root of unity. We diagonalize the action of $G$ and denote the corresponding coordinates on $\mathbb{C}^3$ by $x,y$ and $z$.  The lattice of exponents of Laurent monomials in $x, y, z$ will be denoted by $L=\mathbb{Z}^3$ and the dual lattice by $L^{\vee}$. Associate a vector $v_g=\dfrac{1}{r}(\gamma_1,\gamma_2,\gamma_3)$ to each group element $g=\mbox{diag}(\varepsilon^{\gamma_1},\varepsilon^{\gamma_2},\varepsilon^{\gamma_3})$, define the lattice $N:=L^{\vee}+\sum\limits_{g \in G}\mathbb{Z}\cdot v_g$ (with $N_{\mathbb{R}}=N\otimes_{\mathbb{Z}}\mathbb{R}$) and use $M :=\Hom(N, \mathbb{Z})$ for the dual lattice of $G$-invariant Laurent monomials. The categorical quotient $X=\mbox{Spec}~ \mathbb{C}[x,y,z]^G$ is the toric variety $\mbox{Spec}~ \mathbb{C}[\sigma^{\vee}
\cap M]$ with the cone $\sigma$  being the positive octant $\sigma=\mathbb{R}_{\geq 0}e_i \subset N_{\mathbb{R}}$. 

\begin{defn}
	The \textbf{junior simplex} $\triangle \subset N_{\mathbb{R}}$ is the triangle with vertices $e_x=(1,0,0), e_y=(0,1,0)$ and $e_z=(0,0,1)$. It contains the lattice points $\dfrac{1}{r}(\gamma_1,\gamma_2,\gamma_3)$ with $\gamma_1+\gamma_2+\gamma_3=1, \gamma_i\geq 0$. 
\end{defn}

A subdivision of the cone $\sigma$ gives rise to a fan $\Sigma$ and, hence, a toric variety $X_\Sigma$ together with a birational map $X_\Sigma\rightarrow X$. A triangle inside the junior simplex is called \textbf{basic} in case the pyramid with this triangle as the base and origin as the apex has volume $1$. If a fan $\Sigma$ gives rise to a partition of the junior simplex into basic triangles, then the corresponding map $X_{\Sigma}\rightarrow X$ is a crepant resolution of singularities. Notice that such a fan $\Sigma$ is uniquely determined by the associated triangulation of the junior simplex  into basic triangles (slightly abusing notation we will refer to such a triangulation as $\Sigma$ as well).

\begin{defn}
	A \textbf{G-cluster} is a $G$-invariant zero-dimensional subscheme $Z\subset \mathbb{C}^n$ for which $H^0(\mathcal{O}_Z)$ is isomorphic to the regular representation of $G$  as a $\mathbb{C}[G]$-module. The \textbf{$G\operatorname{-}$Hilbert scheme} is the variety $Y=G\operatorname{-}\mbox{Hilb}(\mathbb{C}^n)$, which is the fine moduli space parameterizing $G$-clusters.
\end{defn}
The $G\operatorname{-}$Hilbert scheme is a toric variety and for $G\subset SL_2(\mathbb{C})$ or $SL_3(\mathbb{C})$ the map $Y\rightarrow X$ is a crepant resolution of singularities. The partition of the junior simplex into basic triangles for finite abelian  $G\subset SL_3(\mathbb{C})$, giving rise to the fan of $Y$, can be computed according to the $3$-step procedure below (see \cite{Nak} and \cite{C1}). Prior to giving the algorithm a definition is due.

\begin{defn}
	Let $r$ and $a$ be coprime positive integers with $r>a$. The
	\textbf{Hirzebruch–Jung continued fraction} of $\dfrac{r}{a}$ is the expression
	$$\dfrac{r}{a}=a_1-\dfrac{1}{a_2-\dfrac{1}{a_3-\hdots}}.$$
	We will also refer to the collection of numbers $[a_1:a_2:\hdots:a_k]$ as the 	\textbf{Hirzebruch–Jung sequence} of $\dfrac{r}{a}$ and denote it by $HJ(r,a)$.
\end{defn}
\begin{rmk}
	Each number $a_i$ in a  Hirzebruch–Jung sequence $[a_1:a_2:\hdots:a_k]$ is greater than or equal to $2$.
\end{rmk}

\begin{ex}
	Let $r\geq 2$ be an integer. The Hirzebruch-Jung continued fraction for $(r,r-1)$ is 	$$\dfrac{r}{r-1}=2-\dfrac{1}{2-\dfrac{1}{2-_{\ddots}\underset{-\dfrac{1}{2}}{}}}$$ with the corresponding sequence $HJ(r,r-1)=[2:2:\hdots:2]$ consisting of an $(r-1)$-tuple of twos.
	\label{Hirz}
\end{ex}

Next we recall an algorithm for producing the partition of $\triangle$ into basic triangles corresponding to $Y$.

\begin{enumerate}
	\item [\textbf{Step} $1$.] Draw line segments connecting the vertices of $\triangle$ to lattice points on the boundary of the convex hull of $\triangle\setminus\{e_x,e_y,e_z\}$ in such a way that the line segments do not cross the interior of $Conv(\triangle\setminus\{e_x,e_y,e_z\})$. Let the Hirzebruch-Jung sequence at a vertex $\zeta\in\{e_x,e_y,e_z\}$   be $HJ_{\zeta}=[k_1:k_2:\hdots :k_s]$. There will be $s+2$ line segments $L^\zeta_0,L^\zeta_1,\hdots,L^\zeta_s,L^\zeta_{s+1}$ emanating from $\zeta$ with $L^\zeta_0$ and $L^\zeta_{s+1}$ being parts of edges between $\zeta$ and the remaining two vertices of $\triangle$. Moreover, $L^\zeta_{i+1}=k_iL^\zeta_i-L^\zeta_{i-1}$ and we label the $i^{th}$ segment with $k_i$ (the edges $L^\zeta_0$ and $L^\zeta_{s+1}$ on the boundary of $\triangle$ receive no label).

	\item [\textbf{Step} $2$.] Extend the lines until they are 'defeated' according to the following rule: when lines meet at a point, the line with greatest label extends and its label value is reduced by $1$ for every 'rival' it defeats; lines meeting with equal labels all terminate at that point. 
	
	\item [\textbf{Step} $3$.] Draw $k-1$ lines parallel to the sides of each \textbf{regular} triangle of side $k$ (lattice triangle with $k+1$ lattice points on each edge) to produce its
	regular tesselation into $k^2$ basic triangles. 
	\label{Algo}
\end{enumerate}
\begin{ex} Let $G=\mathbb{Z}/15\mathbb{Z}\subset SL_3(\mathbb{C})$ with $v_1=\dfrac{1}{15}(1,2,12)$. The exceptional divisors are (the rays corresponding to) the vectors with endpoints $v_k=\dfrac{1}{15}(k_{15},2k_{15},12k_{15})\in \triangle$, i.e. $k_{15}+2k_{15}+12k_{15}=15$ (we use the notation $s_{15}$ for $s$ modulo $15$):

	\begin{align*}
		&E_1=\dfrac{1}{15}(1,2,12), &E_2=\dfrac{1}{15}(2,4,9)\hspace{3in}\\
		&E_3=\dfrac{1}{15}(3,6,6), &E_4=\dfrac{1}{15}(4,8,3)\hspace{3in}\\
		&E_5=\dfrac{1}{15}(5,10,0), &E_6=\dfrac{1}{15}(8,1,6)\hspace{3in}\\
		&E_7=\dfrac{1}{15}(9,3,3), &E_8=\dfrac{1}{15}(10,5,0).\hspace{2.86in}\\	
	\end{align*}

\begin{enumerate}
	\item [\textbf{Step} $1$.] We compute the Hirzebruch-Jung sequences:
	\begin{align*}
		&\dfrac{1}{15}(2,12)\sim \dfrac{1}{15}(1,6)\sim \dfrac{1}{5}(1,2), ~\dfrac{5}{2}=3-\dfrac{1}{2},~HJ_x=HJ(15,2)=[3:2];\hspace{2in}\\
		&\dfrac{1}{15}(12,1)\sim \dfrac{1}{5}(4,1)\sim\dfrac{1}{5}(1,4),~\dfrac{5}{4}=2-\dfrac{1}{2-\dfrac{1}{2-\dfrac{1}{2}}},~HJ_y=HJ(5,4)=[2:2:2:2];\\
		&\dfrac{15}{2}=8-\dfrac{1}{2},~HJ_z=HJ(15,2)=[8:2] \mbox{ and draw}
	\end{align*}

\begin{figure}[htbp!]
		\begin{center}
			
			\begin{tikzpicture}[scale=0.55]
				\draw (0,0) -- (15,0) -- (7.5,13) -- (0,0) -- cycle;
				\draw (0,0) -- (37.5/15,13/15);

				\draw (15,0)  -- (37.5/15,13/15);
				\draw (15,0)  -- (75/15,26/15);
				\draw (15,0)  -- (112.5/15,39/15);
				\draw (15,0)  -- (150/15,52/15);
				
				\draw (7.5,13) -- (5,6.93333);
				\draw (7.5,13) -- (7.5,7.8);
				
				\draw (0,0) -- (5,6.93333);

				\node at (37.5/15,13/15) {$\bullet$};
				\node at (75/15,26/15) {$\bullet$};
				\node at (112.5/15,39/15) {$\bullet$};
				\node at (150/15,52/15) {$\bullet$};
				\node at (187.5/15,65/15) {$\bullet$};
				
				\node at (37.5/15-0.4,13/15+0.4) {$E_1$};
				\node at (75/15-0.5,26/15+0.4) {$E_2$};
				\node at (112.5/15-0.5,39/15+0.3) {$E_3$};
				\node at (150/15-0.5,52/15+0.3) {$E_4$};
				\node at (187.5/15-0.6,65/15+0.2) {$E_5$};
				
				\node at (5,6.93333) {$\bullet$};
				\node at (7.5,7.8) {$\bullet$};
				\node at (10,8.66) {$\bullet$};
				
				\node at (5-0.2,6.93333+0.5) {$E_6$};
				\node at (7,7.8+0.2) {$E_7$};
				\node at (9.5,8.66+0.2) {$E_8$};
				
				\node at (1.3,0.75) {$\mathsmaller{8}$};
			
				\node at (3,4.5) {$\mathsmaller{2}$};
				\node at (6.04,10) {$\mathsmaller{3}$};
				
				\node at (7.7,10) {$\mathsmaller{2}$};
				
				\node at (12.4,2.1) {$\mathsmaller{2}$};
				\node at (10.9,1.7) {$\mathsmaller{2}$};
				\node at (9,1.3) {$\mathsmaller{2}$};
				\node at (7.5,0.75) {$\mathsmaller{2}$};
				\node at (-0.2,-0.2) {$e_z$};
				\node at (15.37,-0.2) {$e_y$};
				\node at (7.5,13.3) {$e_x$};	
			\end{tikzpicture}
			\caption{Step $1$ of the triangulation algorithm for $G=\dfrac{1}{15}(1,2,12)$}
		\end{center}
	\end{figure}
	\newpage
	\item [\textbf{Step} $2$.] Extending the lines according to their labels.
	\newline
	 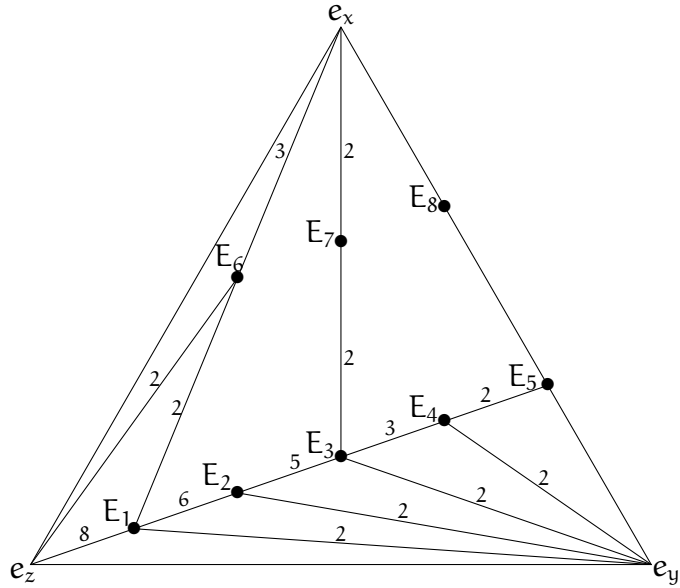
\begin{figure}[htbp!]
		\begin{center}
			
			\begin{tikzpicture}[scale=0.55]
				\draw (0,0) -- (15,0) -- (7.5,13) -- (0,0) -- cycle;
				\draw (0,0) -- (37.5/15,13/15);
				\draw (37.5/15,13/15) -- (75/15,26/15);
				\draw (75/15,26/15) -- (112.5/15,39/15);
				\draw (112.5/15,39/15) -- (150/15,52/15);
				\draw (150/15,52/15) -- (187.5/15,65/15);
				
				\draw (15,0)  -- (37.5/15,13/15);
				\draw (15,0)  -- (75/15,26/15);
				\draw (15,0)  -- (112.5/15,39/15);
				\draw (15,0)  -- (150/15,52/15);
				
				\draw (7.5,13) -- (5,6.93333);
				\draw (7.5,13) -- (7.5,7.8);
				
				\draw (0,0) -- (5,6.93333);

				\draw (5,6.93333)  -- (37.5/15,13/15);
				\draw (7.5,7.8)  -- (112.5/15,39/15);
				
				\node at (37.5/15,13/15) {$\bullet$};
				\node at (75/15,26/15) {$\bullet$};
				\node at (112.5/15,39/15) {$\bullet$};
				\node at (150/15,52/15) {$\bullet$};
				\node at (187.5/15,65/15) {$\bullet$};
				
				\node at (37.5/15-0.4,13/15+0.4) {$E_1$};
				\node at (75/15-0.5,26/15+0.4) {$E_2$};
				\node at (112.5/15-0.5,39/15+0.3) {$E_3$};
				\node at (150/15-0.5,52/15+0.3) {$E_4$};
				\node at (187.5/15-0.6,65/15+0.2) {$E_5$};
				
				\node at (5,6.93333) {$\bullet$};
				\node at (7.5,7.8) {$\bullet$};
				\node at (10,8.66) {$\bullet$};
				
				\node at (5-0.2,6.93333+0.5) {$E_6$};
				\node at (7,7.8+0.2) {$E_7$};
				\node at (9.5,8.66+0.2) {$E_8$};
				
				\node at (1.3,0.75) {$\mathsmaller{8}$};
				\node at (3.7,1.6) {$\mathsmaller{6}$};
				\node at (6.4,2.5) {$\mathsmaller{5}$};
				\node at (8.7,3.3) {$\mathsmaller{3}$};
				\node at (11,4.1) {$\mathsmaller{2}$};
				\node at (3,4.5) {$\mathsmaller{2}$};
				\node at (6.04,10) {$\mathsmaller{3}$};
				\node at (3.53,3.8) {$\mathsmaller{2}$};
				\node at (7.7,10) {$\mathsmaller{2}$};
				\node at (7.7,5) {$\mathsmaller{2}$};	
				\node at (12.4,2.1) {$\mathsmaller{2}$};
				\node at (10.9,1.7) {$\mathsmaller{2}$};
				\node at (9,1.3) {$\mathsmaller{2}$};
				\node at (7.5,0.75) {$\mathsmaller{2}$};
				\node at (-0.2,-0.2) {$e_z$};
				\node at (15.37,-0.2) {$e_y$};
				\node at (7.5,13.3) {$e_x$};	
			\end{tikzpicture}
			\caption{Step $2$ of the triangulation algorithm for $G=\dfrac{1}{15}(1,2,12)$}
		\end{center}
	\end{figure}

\item [\textbf{Step} $3$.] Subdividing regular triangles into basic.  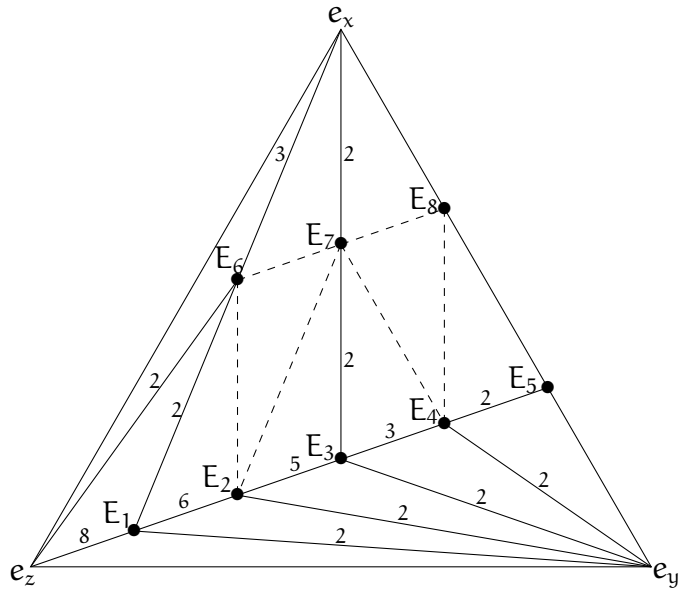
\begin{figure}[htbp!]
	\begin{center}
		
		\begin{tikzpicture}[scale=0.55]
			\draw (0,0) -- (15,0) -- (7.5,13) -- (0,0) -- cycle;
			\draw (0,0) -- (37.5/15,13/15);
			\draw (37.5/15,13/15) -- (75/15,26/15);
			\draw (75/15,26/15) -- (112.5/15,39/15);
			\draw (112.5/15,39/15) -- (150/15,52/15);
			\draw (150/15,52/15) -- (187.5/15,65/15);
			
			\draw (15,0)  -- (37.5/15,13/15);
			\draw (15,0)  -- (75/15,26/15);
			\draw (15,0)  -- (112.5/15,39/15);
			\draw (15,0)  -- (150/15,52/15);
			
			\draw (7.5,13) -- (5,6.93333);
			\draw (7.5,13) -- (7.5,7.8);
			
			\draw (0,0) -- (5,6.93333);

			\draw (5,6.93333)  -- (37.5/15,13/15);
			\draw (7.5,7.8)  -- (112.5/15,39/15);
			
			\draw [dashed] (5,6.93333) -- (7.5,7.8);
			\draw [dashed] (7.5,7.8) -- (10,8.66);
			\draw [dashed] (7.5,7.8) -- (75/15,26/15);
			\draw [dashed] (7.5,7.8) -- (150/15,52/15);
			\draw [dashed] (5,6.93333) -- (75/15,26/15);
			\draw [dashed] (10,8.66) -- (150/15,52/15);
			
			\node at (37.5/15,13/15) {$\bullet$};
			\node at (75/15,26/15) {$\bullet$};
			\node at (112.5/15,39/15) {$\bullet$};
			\node at (150/15,52/15) {$\bullet$};
			\node at (187.5/15,65/15) {$\bullet$};
			
			\node at (37.5/15-0.4,13/15+0.4) {$E_1$};
			\node at (75/15-0.5,26/15+0.4) {$E_2$};
			\node at (112.5/15-0.5,39/15+0.3) {$E_3$};
			\node at (150/15-0.5,52/15+0.3) {$E_4$};
			\node at (187.5/15-0.6,65/15+0.2) {$E_5$};
			
			\node at (5,6.93333) {$\bullet$};
			\node at (7.5,7.8) {$\bullet$};
			\node at (10,8.66) {$\bullet$};
			
			\node at (5-0.2,6.93333+0.5) {$E_6$};
			\node at (7,7.8+0.2) {$E_7$};
			\node at (9.5,8.66+0.2) {$E_8$};
			
			\node at (1.3,0.75) {$\mathsmaller{8}$};
			\node at (3.7,1.6) {$\mathsmaller{6}$};
			\node at (6.4,2.5) {$\mathsmaller{5}$};
			\node at (8.7,3.3) {$\mathsmaller{3}$};
			\node at (11,4.1) {$\mathsmaller{2}$};
			\node at (3,4.5) {$\mathsmaller{2}$};
			\node at (6.04,10) {$\mathsmaller{3}$};
			\node at (3.53,3.8) {$\mathsmaller{2}$};
			\node at (7.7,10) {$\mathsmaller{2}$};
			\node at (7.7,5) {$\mathsmaller{2}$};	
			\node at (12.4,2.1) {$\mathsmaller{2}$};
			\node at (10.9,1.7) {$\mathsmaller{2}$};
			\node at (9,1.3) {$\mathsmaller{2}$};
			\node at (7.5,0.75) {$\mathsmaller{2}$};
			\node at (-0.2,-0.2) {$e_z$};
			\node at (15.37,-0.2) {$e_y$};
			\node at (7.5,13.3) {$e_x$};	
		\end{tikzpicture}
		\caption{Step $3$ of the triangulation algorithm for $G=\dfrac{1}{15}(1,2,12)$}
	\end{center}
\end{figure} 
\end{enumerate}

\label{1212}	
\end{ex}

\subsection{Reid's recipe} Reid’s recipe (see \cite{Re} and \cite{C1}) is an algorithm to construct the cohomological version of the McKay correspondence for abelian subgroups of $SL_3(\mathbb{C})$. It starts with marking internal edges and vertices of the triangulation $\Sigma$ corresponding to $G\operatorname{-}\mbox{Hilb}(\mathbb{C}^3)$
with characters of $G$. For the purposes of this paper only the marking of edges will be required, which we recall below and refer the reader to Section $3$ of \cite{C1} for details. An edge $(e,f)$ in $\Sigma$ is labeled by a character of $G$ according to the following rule. The one-dimensional ray
in $M$ perpendicular to the hyperplane $\langle e,f\rangle$ in $M_{\mathbb{R}}$ has a primitive generator given by exponents of $\dfrac{m_1}{m_2}$, where $m_1$ and $m_2$ are coprime regular monomials. As $M$ is the lattice of $G$-invariant Laurent
monomials, $m_1$ and $m_2$ have the same character $\chi$ with respect to $G$-action. The edge $(e,f)$ is marked with  character $\chi$.
\begin{ex}
We determine the character that marks the edge	$(e_z,E_1)$:

	$\begin{cases}
		c=0\\
		a+2b+12c=0\\
	\end{cases}\Leftrightarrow a=-2b, c=0,$

hence, $m_1=x^2$ and $m_2=y$ with $\chi=\chi_2$.

The computation below shows that the edge	$(E_6,E_7)$ is marked by $\chi=\chi_1$:

$\begin{cases}
	8a+b+6c=0\\
	9a+3b+3c=0\\
\end{cases}\Leftrightarrow \begin{cases}
3a+b+c=0\\
a+c=0\\
\end{cases} \Leftrightarrow  a=-c, b=2c,$

thus, $m_1=x$ and $m_2=y^2z$.
\end{ex}
\begin{figure}[htbp!]
	\begin{center}
		
		\begin{tikzpicture}[scale=0.6]
			\draw (0,0) -- (15,0) -- (7.5,13) -- (0,0) -- cycle;
			\draw (0,0) -- (37.5/15,13/15);
			\draw (37.5/15,13/15) -- (75/15,26/15);
			\draw (75/15,26/15) -- (112.5/15,39/15);
			\draw (112.5/15,39/15) -- (150/15,52/15);
			\draw (150/15,52/15) -- (187.5/15,65/15);
			
			\draw (15,0)  -- (37.5/15,13/15);
			\draw (15,0)  -- (75/15,26/15);
			\draw (15,0)  -- (112.5/15,39/15);
			\draw (15,0)  -- (150/15,52/15);
			
			\draw (7.5,13) -- (5,6.93333);
			\draw (7.5,13) -- (7.5,7.8);
			
			\draw (0,0) -- (5,6.93333);
			\draw [dashed] (5,6.93333) -- (7.5,7.8);
			\draw [dashed] (7.5,7.8) -- (10,8.66);
			\draw [dashed] (7.5,7.8) -- (75/15,26/15);
			\draw [dashed] (7.5,7.8) -- (150/15,52/15);
			\draw [dashed] (5,6.93333) -- (75/15,26/15);
			\draw [dashed] (10,8.66) -- (150/15,52/15);
			
			\draw (5,6.93333)  -- (37.5/15,13/15);
			\draw (7.5,7.8)  -- (112.5/15,39/15);
			
			\node at (37.5/15,13/15) {$\bullet$};
			\node at (75/15,26/15) {$\bullet$};
			\node at (112.5/15,39/15) {$\bullet$};
			\node at (150/15,52/15) {$\bullet$};
			\node at (187.5/15,65/15) {$\bullet$};
			
			\node at (37.5/15-0.4,13/15+0.4) {$E_1$};
			\node at (75/15-0.5,26/15+0.4) {$E_2$};
			\node at (112.5/15-0.5,39/15+0.3) {$E_3$};
			\node at (150/15-0.5,52/15+0.3) {$E_4$};
			\node at (187.5/15-0.6,65/15+0.2) {$E_5$};
			
			\node at (5,6.93333) {$\bullet$};
			\node at (7.5,7.8) {$\bullet$};
			\node at (10,8.66) {$\bullet$};
			
			\node at (5-0.2,6.93333+0.5) {$E_6$};
		\node at (7,7.8+0.2) {$E_7$};
		\node at (9.5,8.66+0.2) {$E_8$};
			
			\node at (1.3,0.75) {$\mathsmaller{\chi_2}$};
			\node at (3.7,1.6) {$\mathsmaller{\chi_2}$};
			\node at (6.4,2.5) {$\mathsmaller{\chi_2}$};
			\node at (8.7,3.3) {$\mathsmaller{\chi_2}$};
			\node at (11,4.1) {$\mathsmaller{\chi_2}$};
			\node at (3.3,5) {$\mathsmaller{\chi_1}$};
			\node at (6.4,7.7) {$\mathsmaller{\chi_1}$};
			\node at (8.7,8.43) {$\mathsmaller{\chi_1}$};
			\node at (7.9,10) {$\mathsmaller{\chi_3}$};	
			\node at (8.9,6) {$\mathsmaller{\chi_3}$};
			\node at (12.7,2) {$\mathsmaller{\chi_3}$};
			\node at (37.5/15+3,13/15) {$\mathsmaller{\chi_{12}}$};
			\node at (37.5/15+6.2,13/15+0.5) {$\mathsmaller{\chi_{9}}$};
			\node at (37.5/15+2.9,13/15+4.5) {$\mathsmaller{\chi_{9}}$};
			\node at (37.5/15+4.1,13/15+3.5) {$\mathsmaller{\chi_{10}}$};
			\node at (37.5/15+1.6,13/15+2.5) {$\mathsmaller{\chi_{12}}$};
			\node at (37.5/15+5.4,13/15+4.5) {$\mathsmaller{\chi_{6}}$};
			\node at (37.5/15+7.9,13/15+4.9) {$\mathsmaller{\chi_{4}}$};
			\node at (112.5/15+3.3,39/15-0.8) {$\mathsmaller{\chi_{6}}$};
			\node at (6.8,10) {$\mathsmaller{\chi_{12}}$};
			\node at (-0.2,-0.2) {$e_z$};
			\node at (15.37,-0.2) {$e_y$};
			\node at (7.5,13.3) {$e_x$};	
		\end{tikzpicture}
		\caption{$\Sigma$ fan and character marking for $G=\dfrac{1}{15}(1,2,12)$}
		\label{JuniorSimplex1212}
	\end{center}
\end{figure}
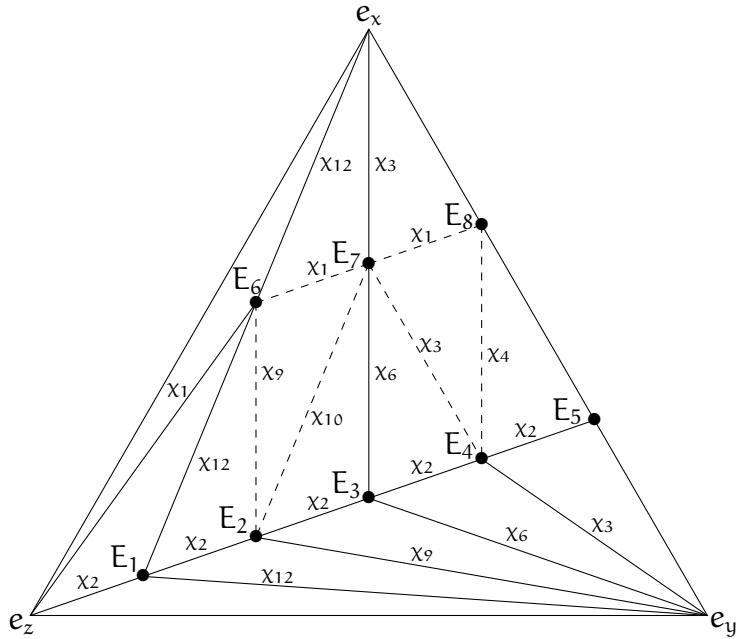
\subsection{Results for abelian subgroups $G\subset SL_3(\mathbb{C})$}

Let $G\subset SL_n({\mathbb{C}})$ be a finite subgroup. The following result appeared in the celebrated paper of Bridgeland, King and Reid (see \cite{BKR}).

\begin{thm}
	Let $G\subset SL_n{\mathbb{C}}$ be a finite subgroup with $n\leq 3$.
	\begin{enumerate}
		\item The variety $G\operatorname{-}\mbox{Hilb}(\mathbb{C}^n)$ is irreducible and the resolution $Y\rightarrow X$ is crepant.
		\item The map $\Psi: D^b_G(\mathbb{C}^n)\rightarrow D^b(Y)$ is an exact equivalence of triangulated categories.
	\end{enumerate}
\label{BKRmain}
\end{thm}
Let $\chi$ be an irreducible representation  of $G$. There are two natural $G$-equivariant sheaves on $\mathbb{C}^n$ associated to $\chi$:
\begin{itemize}
	\item $\widetilde{\chi}:=\chi\otimes \mathcal{O}_{\mathbb{C}^n}$
	\item $\chi^!:=\chi\otimes\mathcal{O}_0$,
\end{itemize}
 where $\mathcal{O}_0=\mathcal{O}_{\mathbb{C}^n}/\mathfrak{m}_0= \mathbb{C}[x_1,x_2,\hdots,x_n]/(x_1,x_2,\hdots,x_n)$ is the structure sheaf of the origin in $\mathbb{C}^n$.

The image $\Psi(\widetilde{\chi}\otimes \mathcal{O}_{\mathbb{C}^n})$ admits a straightforward description. It is isomorphic to 
$L_\chi^{\vee}$, where $L_\chi$ is the corresponding tautological vector bundle. The tautological vector bundles on $Y$ are defined as direct summands in the decomposition  $$p_*(\mathcal{O}_{\mathcal{Z}})=\bigoplus\mathcal{L}_\chi\otimes\chi,$$ with respect to the trivial $G$-action on $Y$ (here $\mathcal{Z}\subset Y\times \mathbb{C}^n$ is the universal subscheme and $p$ the projection on $Y$).

However, the images of skyscraper sheaves $\chi^!$ are more complicated to describe. In case of abelian $G$ the main result of \cite{CCL} provides such a description.
\begin{thm}
\label{CCLthm}
Let $G\subset SL_3(\mathbb{C})$ be a finite abelian subgroup and let
$\chi$ be an irreducible representation of $G$. Then $H^i(\chi^!)=0$ unless $i \in \{0,-1,-2\}$. Moreover, one of the following holds:
 {\renewcommand{\arraystretch}{1.8}	\begin{table}[ht]
		\begin{center}
			\resizebox{16cm}{!}{
				\begin{tabular}{ |c|c|c|c|} 
					\hline
					Reid's recipe& $H^{-2}(\Psi(\chi^!))$& $H^{-1}(\Psi(\chi^!))$& $H^{0}(\Psi(\chi^!))$ \\ 
					\hline
					$\chi$ marks a single divisor $E$& $0$& $0$& $\mathcal{L}^{-1}_{\chi}\otimes\mathcal{O}_E$ \\ 
					\hline
					$\chi$ marks a single curve $C$& $0$& $0$& $\mathcal{L}^{-1}_{\chi}\otimes\mathcal{O}_C$ \\ 
					\hline
					\pbox{13cm}{$\chi$ marks a chain of divisors\\ starting at $E$ and terminating at $F$}& $0$& $\mathcal{L}^{-1}_{\chi}(-E-F)\otimes\mathcal{O}_Z$& $0$ \\ 
					\hline
					\pbox{13cm}{$\chi$ marks three chains of divisors, starting at \\ $E_x$, $E_y$ and $E_z$ and meeting at a divisor $P$}& $0$& $\mathcal{L}^{-1}_{\chi}(-E_x-E_y-E_z)\otimes\mathcal{O}_{\mathcal{V}_Z}$& $0$ \\ 
					\hline
					$\chi=\chi_0$ & $w_{ZF_2}$& $w_{ZF_1}(ZF_2)$& $0$ \\ 
					\hline
			\end{tabular}}
			\caption{Observed values}
		\end{center}
\end{table}}
\label{LogvMain}	
\end{thm}

\begin{rmk}
		Apriori, each object $\Psi(\chi^!)$ is an abstract complex in $D^b(Y)$. It follows from Theorem \ref{CCLthm}  that for every nontrivial character $\chi$ the object $\Psi(\chi^!)$ is a pure sheaf (i.e. some shift of a coherent sheaf).
\end{rmk} 
\section{Quivers}
The next ingredient that we need is McKay quivers. A good reference for basic concepts of quivers and representations thereof is the book \cite{DW}. We also invite the reader to look in the paper \cite{IN} for the exposition on realizations of $G\operatorname{-}$Hilbert schemes as moduli spaces of McKay quiver representations. 
\subsection{Generalities}
\begin{defn}
	A \textbf{quiver} $Q=(Q_0,Q_1)$ is a finite  directed graph with finitely many vertices enumerated by the set $Q_0$ and finitely many edges indexed by $Q_1$. Each edge is uniquely determined by the pair of vertices it connects, which we will denote by $t(a)$ and $h(a)$ standing for 'tail' and 'head', respectively.
\end{defn}
\begin{defn}
	A \textbf{path} $p$ in a quiver $Q = (Q_0,Q_1)$
	is a sequence $a_{\ell} a_{\ell-1}\hdots a_1$ of arrows in $Q_1$ such that $t(a_{i+1}) = h(a_i)$ for $i = 1, 2,\hdots ,\ell-1$. In addition, for every vertex $x\in Q_0$ we introduce a path $e_x$.
	
	The \textbf{path algebra} $\mathcal{P}_Q$ is a $\mathbb{C}$-algebra with a basis labeled
	by all paths in $Q$. The multiplication in $\mathcal{P}_Q$ is given by $$p\cdot q:=\begin{cases}
		pq,\hspace{0.3in} \mbox{if } t(p)=h(q)\\
		0,\hspace{0.4in} \mbox{otherwise},
	\end{cases}$$
	where $pq$ stands for the concatenation of paths subject to the conventions that $pe_x = p$ if $t(p) = x$, and $e_xp = p$ if $h(p) = x$.
\end{defn}

\begin{rmk}
	Notice that $\mathcal{P}_Q$ is of finite dimension over $\mathbb{C}$ if and only if $Q$ has no oriented cycles.
	The path algebra has a natural grading by path length with the subring of grade zero spanned
	by the trivial paths $e_x$
	for $x \in Q_0$. It is a semisimple ring, in which the elements $e_x$ are orthogonal idempotents.

\end{rmk}
\begin{defn}
	A \textbf{representation of a quiver} $Q$ consists of a collection of vector spaces $\{V_i\}_{i\in Q_0}$ and linear
	homomorphisms $\alpha_a\in\mbox{Hom}_{\mathbb{C}}(V_{ta},V_{ha})$
	for each arrow $a \in Q_1$.
\end{defn}
Such representations form a category with morphisms being collections of $\mathbb{C}$-linear maps $\psi_i: V_i\rightarrow W_i$ for all $i\in Q_0$ such that the diagrams

\[\begin{tikzcd}
	{V_{ta}} && {V_{ha}} \\
	\\
	{W_{ta}} && {W_{ha}}
	\arrow["{\alpha_a}", from=1-1, to=1-3]
	\arrow["{\psi_{ta}}"', from=1-1, to=3-1]
	\arrow["{\psi_{ha}}", from=1-3, to=3-3]
	\arrow["{\alpha'_a}", from=3-1, to=3-3]
\end{tikzcd}\]
commute. This category will be denoted by $\mbox{Rep}_{\mathbb{C}}(Q)$.
\begin{thm}
	The category $\mbox{Rep}_{\mathbb{C}}(Q)$ is equivalent to the category of finitely-generated left
	$\mathcal{P}_Q$-modules. In particular, $\mbox{Rep}_{\mathbb{C}}(Q)$ is an abelian category.
	\label{RepsPathAlg}
\end{thm}

Often, the algebra of interest is not the path algebra of a quiver $Q$, its quotient by an ideal of relations. A relation in $\mathcal{P}_Q$ is a $\mathbb{C}$-linear combination of paths of length at least two, each with the same head and the same tail. A quiver with relations is a quiver $Q$ together with a finite set of relations  $\mathcal{R}$. A representation of such a quiver is a representation of $Q$ where any composition of maps indexed by subsequent edges in a relation vanishes (is a zero map). Any finite set of relations $\mathcal{R}$ in $Q$ determines a two-sided ideal $\mathcal{I}_\mathcal{R}\subset \mathcal{P}_Q$.
As before, finite-dimensional representations of $(Q, \mathcal{R})$ form a category $\mbox{Rep}_{\mathbb{C}}(Q, \mathcal{R})$. Moreover, the analogue of Theorem \ref{RepsPathAlg} holds true, i.e. there is an equivalence of categories 

\begin{equation}\label{quiverRepRelns}
	\mbox{Rep}_{\mathbb{C}}(Q, \mathcal{R})\simeq \left(\mathcal{P}_Q/\mathcal{I}_\mathcal{R}\right)\operatorname{-} mod. 
\end{equation}

\subsection{McKay Quivers}
Let $G\subset SL(V)$ be a finite subgroup.
\begin{defn}
	The \textbf{McKay quiver} $Q(G,\mathbb{C}^n)$ is the quiver whose vertices are enumerated by irreducible representations of $G$ with $\dim(\Hom_G(\chi_k\otimes \mathbb{C}^n,\chi_\ell))$ arrows (possibly none) from vertex $k$ to vertex $\ell$. Henceforth we assume that $G$ is abelian. Then all irreducible representations of $G$ are one-dimensional and correspond to characters of $G$: $$\mbox{char}(G):=\{\chi:G\rightarrow \mathbb{C}^*\}.$$  In particular, as a representation of $G$, we have $\mathbb{C}^n=\underset{i=1}{\overset{n}{\bigoplus}}\mathbb{C}\chi_i=:\underset{i=1}{\overset{n}{\bigoplus}}\mathbb{C}e_i$ and let $x_1,x_2,\hdots, x_n \in (\mathbb{C}^n)^*$ be the dual basis to $\{e_1,e_2,\hdots, e_n\}$ with $R = \mathbb{C}[x_1, x_2,\hdots, x_n]$ the coordinate ring of $\mathbb{C}^n$. The chain of isomorphisms $\Hom_G(\chi_k\otimes \mathbb{C}^n,\chi_\ell)\simeq\Hom_G(\chi_k\otimes\underset{i=1}{ \overset{n}{\bigoplus}}\mathbb{C}e_i,\chi_\ell)\simeq\underset{i=1}{ \overset{n}{\bigoplus}}\Hom_G(\chi_k\otimes\mathbb{C}e_i,\chi_\ell)$ provides a natural identification of the maps assigned to the arrows in the McKay quiver $Q(G,\mathbb{C}^n)$ with multiplication by $x_i$'s and, hence, impose the relations corresponding to the commutation of the latter: 
	$$\mathcal{I}:=\langle a_i^{\chi\otimes\chi_j}a_j^{\chi}-a_j^{\chi\otimes\chi_i}a_i^{\chi}~|~\chi\in \mbox{char}(G), 1\leq i,j\leq n \rangle\subset \mathbb{C}Q(G,\mathbb{C}^n)$$
(for every vertex $\chi \in Q_0$ there are $n$ arrows with head at $\chi$ denoted by  $a_k^{\chi}: \chi\otimes\chi_k\rightarrow \chi$ and label the arrow $a_k^{\chi}$ by the monomial $x_k$).	

\end{defn}
\begin{ex}
	Let $G=\mathbb{Z}/n\mathbb{Z}$ be embedded in $SL(\mathbb{C}^2)$ via mapping $1$ to $\left(\begin{array}{cc}
		\varepsilon & 0 \\
		0 & \varepsilon^{-1}\\
	\end{array}\right)$ with $\varepsilon=e^{2\pi i/n}$ the primitive $n^{th}$ root of unity. There are $n$  irreducible representations of $G$, one-dimensional, to be denoted by $\chi_0,\chi_1,\hdots,\chi_{n-1}$ and $\mathbb{C}^2\simeq\chi_1\oplus\chi_{n-1}$. We label $x_k:=a_{\chi_1}^{\chi_k}$ and $y_k:=a_{\chi_{n-1}}^{\chi_k}$, then  the ideal of relations is $\mathcal{I}=\langle x_iy_i-x_{i+1}y_{i+1}\rangle$.

\begin{center}
	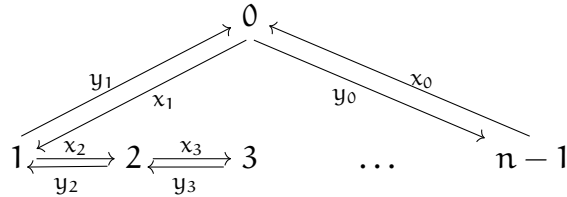
\begin{figure}[h!]
		\begin{tikzpicture}
			\matrix(m)[matrix of math nodes,
			row sep=3em, column sep=2.5em,
			text height=1.5ex, text depth=0.25ex]
			{& & 0\\ 1 & 2 &3 & \hdots  & n-1\\};
			\path[->,font=\scriptsize]
			(m-1-3) +(-0.05,-0.3) edge node[right, yshift=-1mm] {$x_{1}$} (m-2-1)
			(m-2-1) +(0.05,0.3) edge node[left] {$y_{1}$} (m-1-3)  
			(m-1-3) + (0.05,-0.3) edge node[left,yshift=-1mm] {$y_{0}$} (m-2-5)
			(m-2-5) +(-0.05,0.3) edge node[right] {$x_{0}$} (m-1-3) +(0.15,0.3)
			(m-2-1)  edge node[above,yshift=-1mm] {$x_{2}$} (m-2-2)
			(m-2-2) +(-0.35,-0.1) edge node[below] {$y_2$}  (-3.5,-1.05)+(m-2-1) 
			(m-2-2)  edge node[above,yshift=-1mm] {$x_{3}$} (m-2-3)
			(m-2-3) +(-0.35,-0.11) edge node[below] {$y_3$}  (-1.9,-1.05)+(m-2-2) 
			
			;
		\end{tikzpicture}
		\caption{McKay quiver $Q(\mathbb{Z}/n\mathbb{Z},\mathbb{C}^2)$.}
		\label{Slice Quiver}
	\end{figure}
\end{center}
The following Proposition will play an essential role for the construction in Section $5$ (see Lemma $7.5$ and its proof in \cite{C2}).

\begin{prop}
	Let $G \subset GL(\mathbb{C}^n)$ be a finite subgroup. There exists a set of relations $\mathcal{R}$ in the McKay quiver $Q(G,\mathbb{C}^n)$ such that the categories $Rep_{\mathbb{C}}(Q(G,\mathbb{C}^n), \mathcal{R})$ and $Coh_G(\mathbb{C}^n)$ are equivalent.
	\label{GCohandQuiver}
\end{prop}

\end{ex}
\section{Hall algebras}
We will use two versions of the Hall algebra construction. The first variant  appeared in \cite{KV} and was described in detail in \cite{J}. A very good overview of Ringel-Hall algebras over finite fields can be found in \cite{S}. 

Let $\mathcal{C}$ be a $\mathbb{C}$-linear abelian finitary category  (the latter means that all extension groups
between any two objects in $\mathcal{C}$ are finite dimensional).

\subsection{Euler characteristic Hall algebra}

The set of isomorphism classes of objects in $\mathcal{C}$ will be denoted by $\mathcal{C}^{iso}$. The space of functions $Fun(\mathcal{C}^{iso})$  can be made into an associative algebra $\mathcal{H}(\mathcal{C})$, called the \textbf{Hall algebra} of $\mathcal{C}$. Let $\mathcal{G}_{\mathcal{C}}$ be the stack formed by pairs $(A, B)$ of objects of $\mathcal{C}$, where $A$ is a subobject of
$B$, and morphisms of such pairs. There are three morphisms $p_1, p_2, p_3: \mathcal{G}_{\mathcal{C}} \rightarrow \mathcal{C}^{iso}$ 
associating to the pair $(A, B)$ the objects $A, B$ and $B/A$, respectively. The fibers of $p_2$ are algebraic
varieties. The multiplication on $\mathcal{H}(\mathcal{C})$ is given by
$$f\star g:=p_{2*}(p_1^*(f)p_3^*(g)).$$

Let $A\in \mathcal{C}$ be an object and $[C]\in \mathcal{H}(\mathcal{C})$ the characteristic function of $A$. The multiplicity of $[C]$ in $[A]\star[B]$ is $\chi(G^C_{AB})$, where $G^C_{AB}=\{A'\subseteq C~|~A'\simeq A, C/A'\simeq B\}$ and  $\chi$ stands for Euler characteristic with compact support.

The following proposition appearing in \cite{KV} (see Section $3.1$) is a consequence of the fact that the heart of a triangulated category is stable under extensions. It will be an essential ingredient of the main construction in this paper.

\begin{prop}
	Let $\mathcal{C}_1, \mathcal{C}_2$ be two finitary abelian categories, and $\varphi: D^b(\mathcal{C}_1) \rightarrow D^b(\mathcal{C}_2)$ an
	equivalence of triangulated categories. If $A, B, C \in \mathcal{C}_1$ are  such that $\varphi(A), \varphi(B) \in \mathcal{C}_2$
	with $G^C_{AB}\neq \varnothing$, then $\varphi(C)\in \mathcal{C}_2$ and $\varphi$ is an isomorphism of complex varieties $G^C_{AB}\simeq G^{\varphi(C)}_{\varphi(A)\varphi(B)}$.
\end{prop}

\begin{cor}
Let $\mathcal{C}_1, \mathcal{C}_2$ be two finitary abelian categories, and $\varphi: D^b(\mathcal{C}_1) \rightarrow D^b(\mathcal{C}_2)$ an
equivalence of triangulated categories. If $A_1, A_2, \hdots A_n \in \mathcal{C}_1$  are  such that $\varphi(A_1), \varphi(A_2),$ $\hdots, \varphi(A_n) \in D^b(\mathcal{C}_2)$ have cohomology concentrated in a single degree and that degree is the same for all $A_i$, then $\varphi$ induces an isomorphism of Hall algebras $$\mathcal{H}\langle A_1,A_2,\hdots,A_n\rangle\simeq \mathcal{H}\langle \varphi(A_1),\varphi(A_2),\hdots,\varphi(A_n)\rangle.$$	
\label{HallAlgIso}
\end{cor}

\subsection{Ringel-Hall algebras over finite fields}
Assume that $\mathcal{C}$ is a finitary abelian category, such that
\begin{itemize}
	\item gldim$(\mathcal{C})<\infty$;
	\item $|\mbox{Ext}^i(A,B)|<\infty$ for any two objects $A, B \in Ob(\mathcal{C})$ and all $i\geq 0$.
	
\end{itemize} 
\begin{defn}
	The \textbf{multiplicative Euler form} $\langle\cdot,\cdot\rangle:K(\mathcal{C}\times\mathcal{C})\rightarrow \mathbb{C}$ is the form given by $$\langle A,B\rangle:=(\prod\limits_{i=0}^\infty|\mbox{Ext}^i(A,B)|^{(-1)^i})^{1/2}.$$
\end{defn}

Let $\mathcal{P}^C_{A,B}$ be the number of short exact sequences $0\rightarrow B\rightarrow C\rightarrow A\rightarrow 0$ and $$P^C_{A,B}:=\dfrac{\mathcal{P}^C_{A,B}}{|\End(A)||\End(B)|}.$$

Consider the vector space $\mathcal{H}_{fin}(\mathcal{C}):=\underset{A\in \mathcal{C}^{iso}}{\bigoplus}\mathbb{C}[A]$. The following defines  the structure of an associative algebra on $\mathcal{H}_{fin}(\mathcal{C})$:

$$[A]\star_{fin} [B]:=\langle A,B\rangle\sum\limits_C P^C_{A,B}[C]$$

\begin{rmk}
	The unit $i :\mathbb{C}\rightarrow \mathcal{H}_{fin}(\mathcal{C})$ is given by $i(\lambda)= \lambda[0]$, where $0$ is the initial object of $\mathcal{C}$.
\end{rmk}

Let $Q$ be a quiver without oriented cycles, $Rep_{\Bbbk}(Q)$ the category of representations of $Q$ over a finite field $\Bbbk$ and $U_q(\mathfrak{g})$ the quantized enveloping algebra. Here $\mathfrak{g}$ is the Lie algebra associated to the Dynkin diagram formed by $Q$. We denote the simple roots of $\mathfrak{g}$ by $E_i$ and simple representations of $Q$ by $\{S_i\}_{i\in Q_0}$.

The following result was obtained by Ringel and Green (see Theorem $3.15$ in \cite{S}). 

\begin{thm}
Let $v=\sqrt{|\Bbbk|}$. There is an embedding of algebras	$\varphi:U_v(\mathfrak{n}_-)\hookrightarrow \mathcal{H}_{fin}(Rep_{\Bbbk}(Q))$ with $\varphi(E_i)=[S_i]$.  
\label{Ringel}
\end{thm}

\section{Main result}
In this section we will combine the results reviewed in Sections $2-4$ to formulate and establish the main theorem of this paper. A few examples prior to giving the general statement will be helpful. 
 \begin{ex} Consider the groups $G=\mathbb{Z}/6\mathbb{Z}$ and $\widetilde{G}=\mathbb{Z}/7\mathbb{Z}$ with $\nu_1=\dfrac{1}{6}(1,1,4)$ and $\widetilde{\nu}_1=\dfrac{1}{7}(1,1,5)$. The partitions of junior simplices into basic triangles corresponding to $G\operatorname{-}\mbox{Hilb}(\mathbb{C}^3)$ and $\widetilde{G}\operatorname{-}\mbox{Hilb}(\mathbb{C}^3)$ are presented on Figure \ref{JuniorSimplex67}. 

 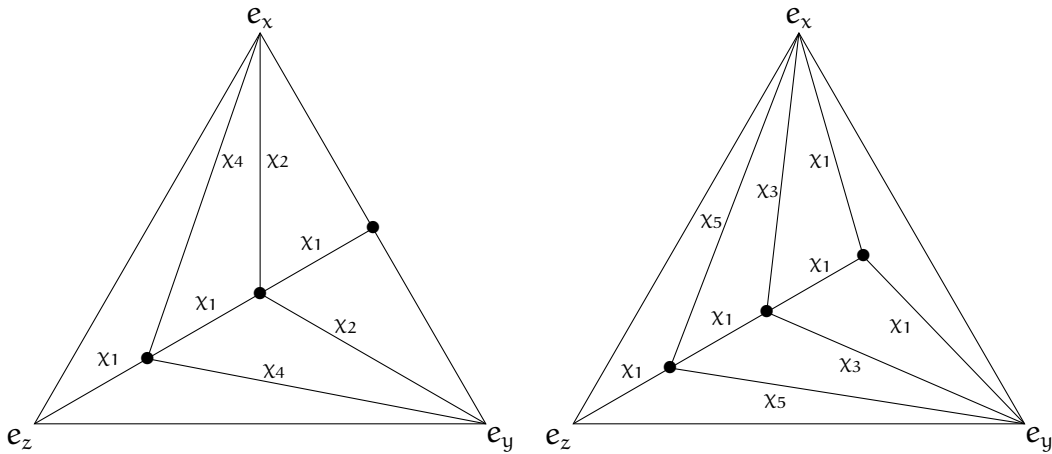
\begin{figure}[htbp!]
 	\begin{center}

 		\begin{tikzpicture}[scale=1]
 		\draw (0,0) -- (6,0) -- (3,5.2) -- (0,0) -- cycle;
 		\draw (0,0) -- (4.5,2.6);
 		\draw (6,0) -- (1.5,5.2/6);
 		\draw (6,0) -- (3,5.2/3);
 		\draw (3,5.2) -- (1.5,5.2/6);
 		\draw (3,5.2) -- (3,5.2/3);
 		\node at (4.5,2.6) {$\bullet$};	
 		\node at (1.5,5.2/6) {$\bullet$};
 		\node at (1,5.2/6) {$\mathsmaller{\chi_1}$};
 		\node at (2.3,1.6) {$\mathsmaller{\chi_1}$};
 		\node at (3.7,2.4) {$\mathsmaller{\chi_1}$};
 		\node at (3.25,3.5) {$\mathsmaller{\chi_2}$};
 		\node at (4.15,1.3) {$\mathsmaller{\chi_2}$};
 		\node at (2.65,3.5) {$\mathsmaller{\chi_4}$};
 		\node at (3.2,0.7) {$\mathsmaller{\chi_4}$};
 		\node at (3,5.2/3) {$\bullet$};	
 		\node at (-0.2,-0.2) {$e_z$};
 		\node at (6.2,-0.2) {$e_y$};
 		\node at (3,5.4) {$e_x$};	
 		\end{tikzpicture}
 	 	\begin{tikzpicture}[scale=1]
 		\draw (0,0) -- (6,0) -- (3,5.2) -- (0,0) -- cycle;
 		\draw (0,0) -- (27/7,15.6/7);
 		\draw (6,0) -- (9/7,5.2/7);
 		\draw (6,0) -- (18/7,10.4/7);
 		\draw (6,0) -- (27/7,15.6/7);
 		\draw (3,5.2) -- (9/7,5.2/7);
 		\draw (3,5.2) -- (18/7,10.4/7);
 		\draw (3,5.2) -- (27/7,15.6/7);
 		\node at (9/7,5.2/7) {$\bullet$};
 		\node at (18/7,10.4/7) {$\bullet$};	
 		\node at (27/7,15.6/7) {$\bullet$};	
 		\node at (0.8,5.2/6-0.17) {$\mathsmaller{\chi_1}$};
 		\node at (2,1.4) {$\mathsmaller{\chi_1}$};
 		\node at (3.3,2.11) {$\mathsmaller{\chi_1}$};
 		\node at (3.3,3.5) {$\mathsmaller{\chi_1}$};
 		\node at (4.37,1.3) {$\mathsmaller{\chi_1}$};
 		\node at (2.6,3.1) {$\mathsmaller{\chi_3}$};
 		\node at (3.7,0.777) {$\mathsmaller{\chi_3}$};
 		\node at (1.85,2.7) {$\mathsmaller{\chi_5}$};
 		\node at (2.7,0.3) {$\mathsmaller{\chi_5}$};
 		\node at (-0.2,-0.2) {$e_z$};
 		\node at (6.2,-0.2) {$e_y$};
 		\node at (3,5.4) {$e_x$};	
 		\end{tikzpicture}
 	
 		\caption{$\Sigma$ fans and character labels for $G=\dfrac{1}{6}(1,1,4)$ and $\widetilde{G}=\dfrac{1}{7}(1,1,5)$}
 		\label{JuniorSimplex67}
 	\end{center}
 \end{figure}
 
 The data on the images of skyscraper sheaves under $\Psi$ appearing in Table \ref{A1TildeTable} is obtained via a direct application of Theorem \ref{LogvMain}. Notice that the McKay quiver $Q(G,\mathbb{C}^3)$ contains a $Q'=\widetilde{A}_1$  subquiver supported on the vertices enumerated by characters $\chi_1$ and $\chi_2$: 
 \[\begin{tikzcd}
	\mathsmaller{1} & \mathsmaller{2}
	\arrow[Rightarrow, from=1-1, to=1-2].
\end{tikzcd}\]
 
 {\renewcommand{\arraystretch}{1.4}	\begin{table}[ht]
 	\begin{center}
 		\resizebox{12cm}{!}{
 			\begin{tabular}{ |c|c|} 
 				\hline
 				$\chi$& $H^{-1}(\Psi(\chi^!))$ \\ 
 				\hline
 				$\chi_1$& $\mathcal{L}^{-1}_{\chi_1}(-E_{z3})\otimes\mathcal{O}_{E_{12}}$ \\ 
 			
 				$\chi_2$& $\mathcal{L}^{-1}_{\chi_2}(-E_{xy})\otimes\mathcal{O}_{E_{2}}$ \\ 
 			
 				$\chi_4$& $\mathcal{L}^{-1}_{\chi_4}(-E_{xy})\otimes\mathcal{O}_{E_{1}}$ \\ 
 				\hline
 		\end{tabular}
 		\quad\quad\quad\quad
			\begin{tabular}{ |c|c|} 
				\hline
				$\chi$& $H^{-1}(\Psi(\chi^!))$ \\ 
				\hline
				$\chi_1$& $\mathcal{L}^{-1}_{\chi_1}(-E_{xyz})\otimes\mathcal{V}_{E_{3}}$ \\ 
			
				$\chi_3$& $\mathcal{L}^{-1}_{\chi_3}(-E_{xy})\otimes\mathcal{O}_{E_{2}}$ \\ 
			
				$\chi_5$& $\mathcal{L}^{-1}_{\chi_5}(-E_{xy})\otimes\mathcal{O}_{E_{1}}$ \\ 
				\hline
		\end{tabular}}
		\caption{Images of $\chi^!_i$'s under $\Psi$ for $G=\dfrac{1}{6}(1,1,4)$ and $\widetilde{G}=\dfrac{1}{7}(1,1,5)$}
		\label{A1TildeTable}
	\end{center}
\end{table}}

 \end{ex}

\begin{ex}
We continue with Example \ref{1212}. Recall that $G=\mathbb{Z}/15\mathbb{Z}$  with $\nu_1=\dfrac{1}{15}(1,2,12)$. See Figure \ref{JuniorSimplex1212} for the partition of junior simplex into basic triangles corresponding to $G\operatorname{-}\mbox{Hilb}(\mathbb{C}^3)$ and marking of edges with characters. This time the McKay quiver $Q(G,\mathbb{C}^3)$ contains a $Q'=\widetilde{A}_2$  subquiver supported on the vertices enumerated by characters $\chi_1, \chi_2$ and $\chi_3$:

  {\renewcommand{\arraystretch}{1.4}	\begin{table}[ht]
 		\begin{center}
 			\hspace{-3in}		\resizebox{5cm}{!}{
 		\begin{tabular}{ |c|c|} 
 					\hline
 					$\chi$& $H^{-1}(\Psi(\chi^!))$ \\ 
 					\hline
 					$\chi_1$& $\mathcal{L}^{-1}_{\chi_1}(-E_{z5})\otimes\mathcal{O}_{E_{1234}}$ \\ 
 					
 					$\chi_2$& $\mathcal{L}^{-1}_{\chi_2}(-E_{z8})\otimes\mathcal{O}_{E_{67}}$ \\ 
 					
 					$\chi_3$& $\mathcal{L}^{-1}_{\chi_3}(-E_{xy})\otimes\mathcal{O}_{E_{47}}$ \\ 
 					\hline
 			\end{tabular}}
 		\vspace{-0.7in}
 		\[\begin{tikzcd}
 			& & &  & &	\mathsmaller{1} & \mathsmaller{2} & \mathsmaller{3}.
 			\arrow[from=1-6, to=1-7]
 			\arrow[from=1-7, to=1-8]
 			\arrow[bend left=40,  from=1-6, to=1-8]
 		\end{tikzcd}\]
 			\caption{Images of $ \chi^!_{1,2,3}$ under $\Psi$ for $G=\dfrac{1}{15}(1,2,12)$}
 		\end{center}
 \end{table}}
\end{ex}

\begin{thm}
Let $r=k+s+1$ with $s\equiv 0 ~(\mbox{mod } k)$ and $s\equiv 0 ~(\mbox{mod } k+1)$. Set $t:=\dfrac{r}{k+1}$ and consider the sequence
\begin{align*}
	&\gamma_{0,t}=1\\
 &\gamma_{n,t}:=n t-n(n-1).
\end{align*}

\begin{enumerate}
	\item If a character $\chi_c$ appears marking an edge, then $c$ satisfies at least one of the following conditions:
	
	\begin{itemize}
		\item $1\leq c\leq k+1$
		\item $c$ is divisible by $k$
		\item $c$ is divisible by $k+1$
	\end{itemize}
	
	\item The images of the $k+1$ skyscraper sheaves corresponding to the first $k+1$ nontrivial characters of $G$ are as presented below.
	
	{\renewcommand{\arraystretch}{1.4}	\begin{table}[ht]
			\begin{center}
				\resizebox{9cm}{!}{
					\begin{tabular}{ |c|c|} 
						\hline
						$\chi$& $H^{-1}(\Psi(\chi^!))$ \\ 
						\hline
						$\chi_1$& $\mathcal{L}^{-1}_{\chi_1}(-E_{z\gamma_{k,t}})\otimes\mathcal{O}_{E_{\gamma_{k-1,t}+2\gamma_{k-1,t}+3\hdots \gamma_{k,t}-1}}$ \\ 
						\hline
						$\chi_2$& $\mathcal{L}^{-1}_{\chi_2}(-E_{z\gamma_{k-1,t}})\otimes\mathcal{O}_{E_{\gamma_{k-2,t}+2\gamma_{k-2,t}+3\hdots \gamma_{k-2,t}-1}}$ \\ 
						\hline
						$\hdots$ & $\hdots$ \\
						\hline
						$\chi_{k}$& $\mathcal{L}^{-1}_{\chi_k}(-E_{z\gamma_{1,t}})\otimes\mathcal{O}_{E_{123\hdots \gamma_{1,t}-1}}$ \\ 
						\hline
						$\chi_{k+1}$& $\mathcal{L}^{-1}_{\chi_{k+1}}(-E_{xy})\otimes\mathcal{O}_{E_{\gamma_{1,t}-1\gamma_{2,t}-1\hdots \gamma_{k,t}-1}}$ \\ 
						\hline
				\end{tabular}}
				\caption{Images of $ \chi^!_{1,2,\hdots,k+1}$ under $\Psi$ for $G=\dfrac{1}{r}(1,k,s)$}
			\end{center}
		\label{CharTable}
	\end{table}}
\end{enumerate}

\label{CohMin1}
\end{thm}
\begin{proof}
	
We follow Reid's recipe (see Section $2$).	Let $\ell=\dfrac{s}{k}$ and compute the Hirzebruch-Jung fractions and sequences at  vertices of $\triangle$ (see Example \ref{Hirz}):
	
		
  \begin{align*}
  	&\dfrac{r}{k}=\ell+2-\dfrac{k-1}{k},~~HJ_z=[\ell+2:\underbrace{2:2:\hdots:2}_{k-1}];\\
  	&\dfrac{1}{r}(s,1)=\dfrac{1}{r/(k+1)}\left(\dfrac{s}{k+1},1\right)=\dfrac{1}{t}(t-1,1)=\dfrac{1}{t}(1,t-1),~~HJ_y=[\underbrace{2:2:\hdots:2}_{t-1}];\\
  	&\dfrac{1}{r}(k,s)=\dfrac{1}{r}\left(1,\dfrac{s}{k}\right)=\dfrac{1}{\dfrac{r}{k+1}}\left(1,\dfrac{s}{k(k+1)}\right)=\dfrac{1}{t}\left(1,\dfrac{t-1}{k}\right),~~HJ_x=[k+1:\underbrace{2:2:\hdots:2}_{(t-1)/k-1}];\\
  \end{align*}
(we have used that $\dfrac{t}{(t-1)/k}=k+1-\dfrac{(t-1)/k-1}{(t-1)/k}$ to evaluate $HJ_X$).
	
	As  $HJ_x=[k+1:\underbrace{2:2:\hdots:2}_{(t-1)/k-1}]$, the line segment connecting $e_x$ and $E_{\gamma_{{k-1},t-1}+1}$ (or $E_{13}$ on Figure \ref{JuniorSimplex1324}) continues until $E_1$. As $\ell+2=s/k+2>2=k+1-(k-1)$ and $s/k+2-t/k-(t-1)=3>2$, 
	it follows that each of the line segments emerging from $e_y$ and each of the line segments emanating from $e_z$ except $L^{e_z}_1$ does not continue beyond the first lattice point on it (see Step $2$ of Algorithm \ref{Algo}). On the other hand $L^{e_z}_1$ defeats all the line segments on its way to $E_{t+1}$, which belongs to the edge $[e_x,e_y]$ of $\triangle$. The remaining edges of triangulation of $\triangle$ (see Step $3$ of the algorithm) come in three families of parallel lines with slopes equal to those of
	
	\begin{itemize}
		\item $[e_z,E_1]$
		\item $[e_x, e_y]$ 
		\item $[e_x, E_{(t+1)/2}]$.
	\end{itemize}
	
	Next we verify that the edges of triangulation of $\triangle$ are marked as in the statement of the theorem. Let $1\leq i\leq k$ and consider the edge $(e_z,E_{{\gamma_{k-i},t}+1})$. Denote $\alpha_i:=\left(1+\dfrac{(r-1)(i-1)}{k}\right)$, then
	
	$\begin{cases}
	c=0\\
	\alpha_i a+(k-i+1)b=0,\\
\end{cases}$

hence, $m_1=x^{k-i}$ and $m_2=y^{\alpha_i}$ with $\chi=\chi_{k-i+1}$ (the solution is $(a,b,c)=(k-i+1,-\alpha_i,0)$).

For each $1\leq i\leq k$ the vertices $E_{{\gamma_{k-i},t}+1},E_{{\gamma_{k-i},t}+2},\hdots,E_{{\gamma_{k-i+1},t}}$ lie on a line  $\ell_i$, moreover, all these lines are parallel and have direction vector $\mathbf{v}=(1,k,r-k-1)$ (see Figure \ref{JuniorSimplex1324}). As the solution $(a,b,c)=(k-i+1,-\alpha_i,0)$ satisfies the equation $\mathbf{v}\cdot(a,b,c)\equiv 0~(\mbox{mod}~ r)$, it follows that all intervals on the line $\ell_i$ are labeled by the character $\chi_{k-i+1}$.

	Now we check that the edge $(e_y,E_{t-1})$ is marked with $\chi_{k+1}$. As $E_{t-1}=\dfrac{1}{r}(t-1,r-k-t,k+1)$, we get the system of equations

$\begin{cases}
	b=0\\
	(t-1)a+(k+1)c=0,\\
\end{cases}$

hence, $m_1=x^{k+1}$ and $m_2=z^{t-1}$ with $\chi=\chi_{k+1}$ (the solution is $(a,b,c)=(k+1,0,1-t)$). The vertices $E_{t-1}, E_{\gamma_{2,t}-1},\hdots,E_{\gamma_{k,t}-1}$ lie on the same line, with the vector connecting any two subsequent vertices being $(t,-t,0)$. As $(k+1)t=(k+1)\dfrac{r}{k+1}=r$, we (inductively) get that the solution $(a,b,c)=(k+1,0,1-t)$ satisfies the equations imposed by $E_{\gamma_{g,t}-1}$ and $E_{\gamma_{g+1,t}-1}$, so each edge $(E_{\gamma_{g,t}-1}, E_{\gamma_{g+1,t}-1})$ is marked by $\chi_{k+1}$. A straightforward computation shows that the edge $(e_x,E_{\gamma_{k,t}})$ is marked with $\chi_{k+1}$ as well. The remaining 'labeling assertions' can be checked similarly using the parallelism of corresponding lines. In particular, notice that
\begin{itemize}
	\item a move by one unit to the right in the family of lines with slope equal to the one of $[e_z,E_1]$  results  in decrease of the corresponding label by $k$; 
	\item a move by one unit to the right in the family of lines with slope equal to the one of $[e_x,e_y]$  results  in decrease of the corresponding label by $k+1$; 
	\item a move by one unit to the right in the family of lines with slope equal to the one of $[e_x,E_{(t+1)/2}]$  results  in decrease of the corresponding label by $k$ to the left and by $k+1$ to the right of the line segment $[e_x,E_{(t+1)/2}]$.  
\end{itemize}

\end{proof}	

\begin{figure}[htbp!]
	\begin{center}
		
		\begin{tikzpicture}[scale=1.1]
			\draw (0,0) -- (15,0) -- (7.5,13) -- (0,0) -- cycle;
			
			\coordinate (Z) at (0,0);
			\coordinate (Y) at (15,0);
			\coordinate (X) at (7.5,13);
			
			\coordinate [label=above left: $E_1$] (E1) at (52.5/28,13/28);
			\coordinate [label=above left: $E_2$] (E2) at (2*52.5/28,2*13/28) ;
			\coordinate [label=above left: $E_3$] (E3) at (3*52.5/28,3*13/28) ;
			\coordinate [label=above left: $E_4$] (E4) at (4*52.5/28,4*13/28) ;
			\coordinate [label=above left: $E_5~$] (E5) at (5*52.5/28,5*13/28) ;
			\coordinate [label=above left: $E_6~$] (E6) at (6*52.5/28,6*13/28) ;
			\coordinate [label=above left: $E_7~$] (E7) at (7*52.5/28,7*13/28) ;
			
			\node at (7*52.5/28+0.3,7*13/28-0.1) {$\ell_1$};
			
			\coordinate [label=above left: $E_8$] (E8) at (105/28,130/28) ;
			\coordinate [label=above left: $E_9$] (E9) at (157.5/28,143/28);
			\coordinate [label=above left: $E_{10}$] (E10) at (210/28,156/28) ;
			\coordinate [label=above left: $E_{11}~$] (E11) at (262.5/28,169/28) ;
			\coordinate [label=above left: $E_{12}~$] (E12) at (315/28,182/28) ;
			
			\node at (315/28+0.3,182/28-0.1) {$\ell_2$};
			
			\coordinate [label=above left : $E_{13}$] (E13) at (157.5/28,247/28) ;
			\coordinate [label=above left: $E_{14}$] (E14) at (210/28,260/28) ;
			\coordinate [label=above left: $E_{15}~$] (E15) at (262.5/28,273/28) ;
			
			\node at (262.5/28+0.3,273/28-0.1) {$\ell_3$};
			
			\node at (E1) {$\bullet$};
			\node at (E2) {$\bullet$};
			\node at (E3) {$\bullet$};
			\node at (E4) {$\bullet$};
			\node at (E5) {$\bullet$};
			\node at (E6) {$\bullet$};
			\node at (E7) {$\bullet$};
			\node at (E8) {$\bullet$};
			\node at (E9) {$\bullet$};
			\node at (E10) {$\bullet$};
			\node at (E11) {$\bullet$};
			\node at (E12) {$\bullet$};
			\node at (E13) {$\bullet$};
			\node at (E14) {$\bullet$};
			\node at (E15) {$\bullet$};
			
			\draw (X)  edge node[right]{$\mathsmaller{\chi_{24}}$} (E13);
			\draw (E8)  edge node[right]{$\mathsmaller{\chi_{24}}$} (E13);
			\draw (E8)  edge node[right]{$\mathsmaller{\chi_{24}}$} (E1);

			\draw (X) edge node[right]{$\mathsmaller{\chi_{4}}$} (E14);
			\draw (E10) edge node[right]{$\mathsmaller{\chi_{12}}$} (E14);
			\draw (E4) edge node[right]{$\mathsmaller{\chi_{12}}$} (E10);
			
			\draw (Z) edge node[above]{$\mathsmaller{\chi_{3}}$} (E1);
			\draw (E1) edge node[above]{$\mathsmaller{\chi_{3}}$} (E2);
			\draw (E2) edge node[above]{$\mathsmaller{\chi_{3}}$} (E3);
			\draw (E3) edge node[above]{$\mathsmaller{\chi_{3}}$} (E4);
			\draw (E4) edge node[above]{$\mathsmaller{\chi_{3}}$} (E5);
			\draw (E5) edge node[above]{$\mathsmaller{\chi_{3}}$} (E6);
			\draw (E6) edge node[above]{$\mathsmaller{\chi_{3}}$} (E7);
			
			\draw (Z)  edge node[above]{$\mathsmaller{\chi_{2}}$} (E8);
			\draw (Z)  edge node[above]{$\mathsmaller{\chi_{1}}$} (E13);
			
			\draw (Y) edge node[above]{$\mathsmaller{\chi_{24}}$} (E1);
			\draw (Y) edge node[above]{$\mathsmaller{\chi_{20}}$} (E2);
			\draw (Y) edge node[above]{$\mathsmaller{\chi_{16}}$} (E3);
			\draw (Y) edge node[above]{$\mathsmaller{\chi_{12}}$} (E4);
			\draw (Y) edge node[above]{$\mathsmaller{\chi_{8}}$} (E5);
			\draw (Y) edge node[above]{$\mathsmaller{\chi_{4}}$} (E6);
			
			\draw[dashed] (E2) edge node[left]{$\mathsmaller{\chi_{20}}$} (E8);
			\draw[dashed] (E3) edge node[left]{$\mathsmaller{\chi_{16}}$} (E9);
			\draw[dashed] (E5) edge node[right]{$\mathsmaller{\chi_{9}}$} (E11);
			\draw[dashed] (E11)  edge node[right]{$\mathsmaller{\chi_{9}}$} (E15);
			\draw[dashed] (E6)  edge node[right]{$\mathsmaller{\chi_{6}}$} (E12);
			
			\draw[dashed] (E13)  edge node[left]{$\mathsmaller{\chi_{16}}$} (E9);

			\draw[dashed] (E8)  edge node[above]{$\mathsmaller{\chi_{2}}$} (E9);
			\draw[dashed] (E9)  edge node[above]{$\mathsmaller{\chi_{2}}$} (E10);
			\draw[dashed] (E10)  edge node[above]{$\mathsmaller{\chi_{2}}$} (E11);
			\draw[dashed] (E11)  edge node[above]{$\mathsmaller{\chi_{2}}$} (E12);
			
			\draw[dashed] (E13) edge node[above]{$\mathsmaller{\chi_{1}}$} (E14);
			\draw[dashed] (E14) edge node[above]{$\mathsmaller{\chi_{1}}$} (E15);
			
			\draw[dashed] (E2) edge node[left]{$\mathsmaller{\chi_{21}}$} (E9);
			\draw[dashed] (E14) edge node[left]{$\mathsmaller{\chi_{21}}$} (E9);

			\draw[dashed] (E6) edge node[right]{$\mathsmaller{\chi_{4}}$} (E11);
			\draw[dashed] (E11) edge node[right]{$\mathsmaller{\chi_{4}}$} (E14);
			
			\draw[dashed] (E3) edge node[left]{$\mathsmaller{\chi_{18}}$}  (E10);
			\draw[dashed] (E5) edge node[right]{$\mathsmaller{\chi_{8}}$} (E10);

			\node  at (-0.2,-0.2) {$e_z$};
			\node  at (15.37,-0.2) {$e_y$};
			\node  at (7.5,13.3) {$e_x$};	
		\end{tikzpicture}
		\caption{$\Sigma$ fans and character marking for $G=\dfrac{1}{28}(1,3,24)$}
		\label{JuniorSimplex1324}
	\end{center}
\end{figure}
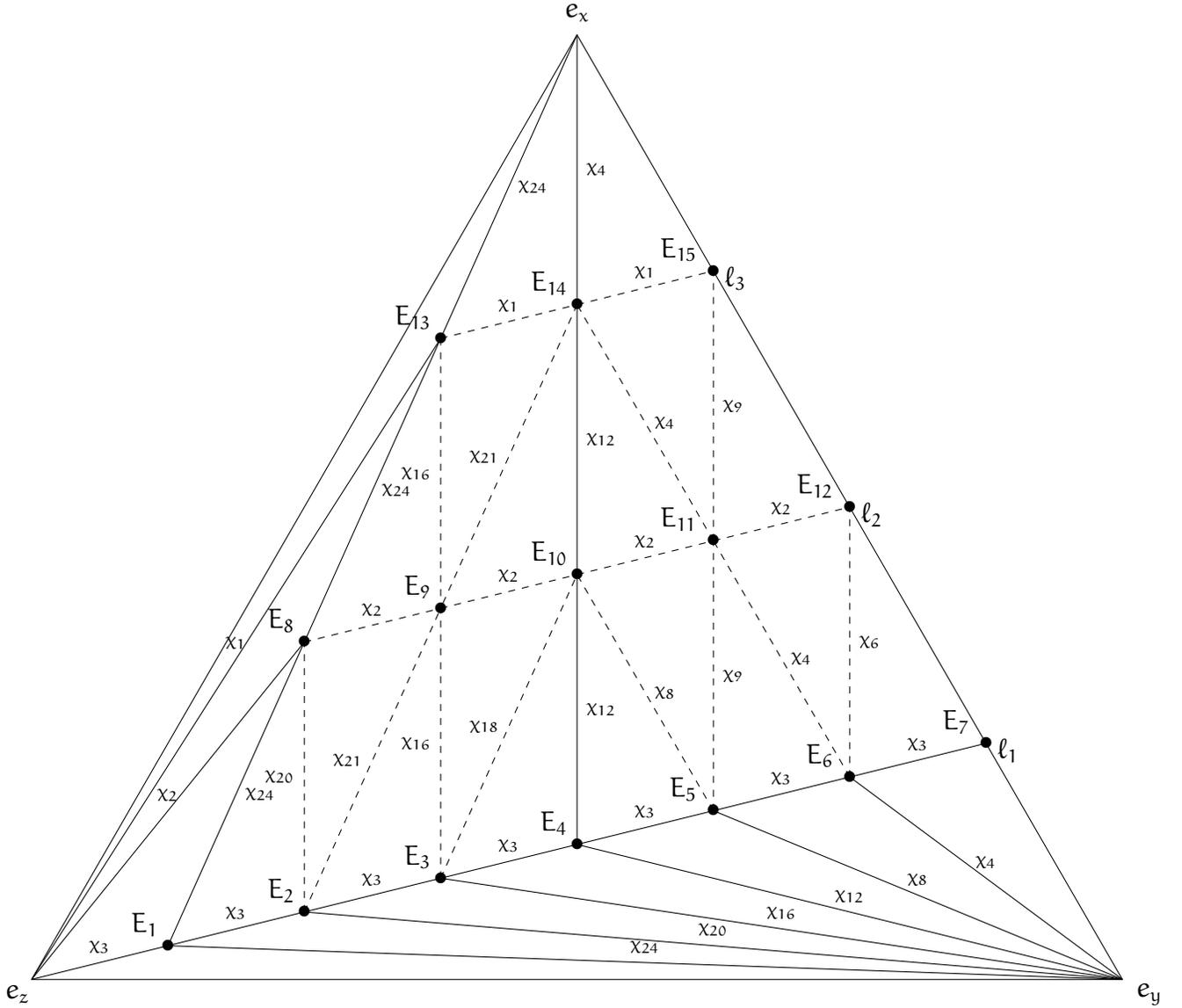
\begin{cor}
	Let $k$ and $s$ satisfy the assumptions of Theorem \ref{CohMin1} with $k$ fixed. Denote the set of characters $\chi$ with  $H^{0}(\Psi(\chi^!))\neq 0$ by $\mathfrak{H}_0$. Then $\lim\limits_{s\rightarrow \infty}\dfrac{|\mathfrak{H}_0|}{r-1}\geq \dfrac{k-1}{k+1}$. In particular, $\lim\limits_{\substack{s\rightarrow \infty\\k\rightarrow \infty}}\dfrac{|\mathfrak{H}_0|}{r-1}=1$.
\label{LimRatio}
\end{cor}
\begin{proof}
	It follows from $(1)$ in Theorem \ref{CohMin1} combined with Theorem \ref{CCLthm} that 
	
	\begin{align*}
	&|\mathfrak{H}_0|\geq r-(k+1)-\dfrac{r-1}{k}-\dfrac{r-1}{k+1}+\dfrac{r-1}{k(k+1)}=\\
	&=\dfrac{sk(k+1)-(k+1)(s+k)-k(s+k)+s+k}{k(k+1)}=\dfrac{(k^2-k)s-2k^2}{k(k+1)}=\dfrac{(k-1)s-2k}{k+1}.\hspace{2in}
	\end{align*}
	Hence, $\lim\limits_{s\rightarrow \infty}\dfrac{|\mathfrak{H}_0|}{r-1}\geq\lim\limits_{s\rightarrow \infty}\dfrac{(k-1)s-2k}{(k+1)(k+s)}=\dfrac{k-1}{k+1}.$
\end{proof}

\begin{rmk}
	For $s\gg k\gg 0$ the cohomology of $\Psi(\chi^!)$ tend to concentrate in degree $0$.
\end{rmk}
\begin{thm}
	Let $r=k+s+1$ with $k$ and $s$ satisfying the assumptions of Theorem \ref{CohMin1} and $s \gg 0$.
	\begin{enumerate}
		\item Let  $\mathfrak{n}_-\subset \widehat{\mathfrak{sl}}_{k+1}(\mathbb{C})$ be a nilpotent subalgebra and $Q'\subset Q(G,\mathbb{C}^3)$ the subquiver supported on  vertices $Q'_0=\{1,2,\hdots,k+1\}$. There is an isomorphism of algebras $$U(\mathfrak{n}_-)\overset{\Theta}{\rightarrow}\mathcal{H}\langle \{\Psi(\chi_i^!)\}_{i\in Q'_0}\rangle.$$	
		\begin{figure}[htbp!]
			\[\begin{tikzcd}
				\mathsmaller{1} & \mathsmaller{2} & \mathsmaller{3} & \mathsmaller{4} & \mathsmaller{\hdots} & \mathsmaller{{k+1}}
				\arrow[from=1-1, to=1-2]
				\arrow[from=1-2, to=1-3]
				\arrow[from=1-3, to=1-4]
				\arrow[from=1-4, to=1-5]
				\arrow[from=1-5, to=1-6]
				\arrow[bend left=40, from=1-1, to=1-6]
			\end{tikzcd}\]\caption{$Q'=\widetilde{A}_{k}\subset Q(G,\mathbb{C}^3)$}
		\end{figure}
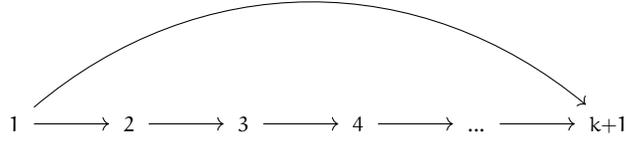
		\item Assume, in addition, that $k=2q+1\geq 5$ is odd, $5\leq n \leq k+1$ and  $\alpha$ satisfies the properties: 
		\begin{itemize}
			\item $\alpha\equiv q ~(\mbox{mod } k)$, $\alpha\equiv -2 ~(\mbox{mod } k+1)$;
			\item $\alpha+k(n-3)+k+1\leq r-1 \Leftrightarrow \alpha\leq s-k^2+3k-1$.
		\end{itemize}
	Let $\mathfrak{n}_-\subset \widehat{\mathfrak{so}}_{2n}(\mathbb{C})$ with $5\leq n \leq k+1$ be a nilpotent subalgebra and $Q'\subset Q(G,\mathbb{C}^3)$  the subquiver supported on  vertices $Q'_0=\{\alpha-k-1, \alpha, \alpha+1, \alpha+k, \alpha+2k, \hdots, \alpha+k(n-3),    \alpha+k(n-3)-1, \alpha+k(n-3)+k+1\}$. There is an isomorphism of algebras $$U(\mathfrak{n}_-)\overset{\Theta}{\rightarrow}\mathcal{H}\langle \{\Psi(\chi_i^!)\}_{i\in Q'_0}\rangle.$$

	\begin{figure}[htbp!]
		\[\begin{tikzcd}
			{\mathsmaller{\alpha+1}} && {} &&& {\mathsmaller{\alpha+k(n-3)-1}} \\
			& \mathsmaller{\alpha} & {\mathsmaller{\alpha+k}} & \mathsmaller{\hdots} & {\mathsmaller{\alpha+k(n-3)}} \\
			\mathsmaller{{\alpha-k-1}} && {} &&& \mathsmaller{{\alpha+kn-2k+1}}
			\arrow[from=2-2, to=2-3]
			\arrow[from=2-2, to=1-1]
			\arrow[from=2-2, to=3-1]
			\arrow[from=2-3, to=2-4]
			\arrow[from=2-4, to=2-5]
			\arrow[from=1-6, to=2-5]
			\arrow[from=3-6, to=2-5]
		\end{tikzcd}\]
		\caption{$Q'=\widetilde{D}_{n+1}\subset Q(G,\mathbb{C}^3)$}
	\end{figure}
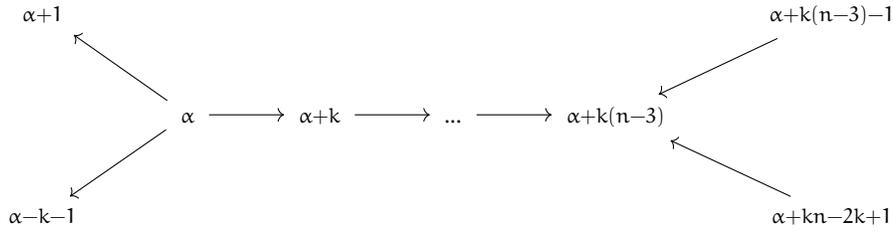
\item Assume that $k>8$ and let $\mathfrak{g}$ be the Lie algebra of type $E_6$ with  $\mathfrak{n}_-\subset \widehat{\mathfrak{g}}$  a nilpotent subalgebra and $Q'\subset Q(G,\mathbb{C}^3)$ the subquiver supported on vertices $Q'_0=\{k+6,2k+6,3k+4,3k+5,3k+6,4k+7,5k+8\}$. There is an isomorphism of algebras $$U(\mathfrak{n}_-)\overset{\Theta}{\rightarrow}\mathcal{H}\langle \{\Psi(\chi_i^!)\}_{i\in Q'_0}\rangle.$$

\vspace{-0.1in}	
\begin{figure}[htbp!]
\[\begin{tikzcd}
	\mathsmaller{{3k+4}} & 	\mathsmaller{{3k+5}} & 	\mathsmaller{{3k+6}} & 	\mathsmaller{{4k+7}} & 	\mathsmaller{{5k+8}} \\
	&& 	\mathsmaller{{2k+6}} \\
	&& 	\mathsmaller{{k+6}}
	\arrow[from=1-1, to=1-2]
	\arrow[from=2-3, to=1-3]
	\arrow[from=1-2, to=1-3]
	\arrow[from=1-4, to=1-3]
	\arrow[from=1-5, to=1-4]
	\arrow[from=3-3, to=2-3]
\end{tikzcd}\]
	\caption{$Q'=\widetilde{E}_{6}\subset Q(G,\mathbb{C}^3)$}
\end{figure}
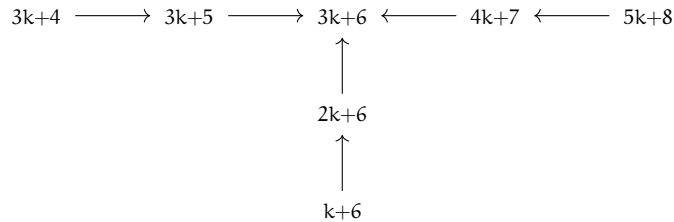
\item Assume that $k>9$ and let $\mathfrak{g}$ be the Lie algebra of type $E_7$ with  $\mathfrak{n}_-\subset \widehat{\mathfrak{g}}$ a nilpotent subalgebra and $Q'\subset Q(G,\mathbb{C}^3)$  the subquiver supported on vertices $Q'_0=\{k+6,2k+3,2k+4,2k+5,2k+6,3k+7,4k+8, 5k+9\}$. There is an isomorphism of algebras $$U(\mathfrak{n}_-)\overset{\Theta}{\rightarrow}\mathcal{H}\langle \{\Psi(\chi_i^!)\}_{i\in Q'_0}\rangle.$$

\vspace{-0.1in}	
\begin{figure}[htbp!]
	\[\begin{tikzcd}
		\mathsmaller{{2k+3}} & \mathsmaller{{2k+4}} & \mathsmaller{{2k+5}} & \mathsmaller{{2k+6}} & \mathsmaller{{3k+7}} & \mathsmaller{{4k+8}} & \mathsmaller{{5k+9}} \\
		&&& \mathsmaller{{k+6}}
		\arrow[from=1-2, to=1-3]
		\arrow[from=2-4, to=1-4]
		\arrow[from=1-3, to=1-4]
		\arrow[from=1-5, to=1-4]
		\arrow[from=1-6, to=1-5]
		\arrow[from=1-7, to=1-6]
		\arrow[from=1-1, to=1-2]
	\end{tikzcd}\]
	\caption{$Q'=\widetilde{E}_{7}\subset Q(G,\mathbb{C}^3)$}
\end{figure}
\item Assume that $k>10$ and let $\mathfrak{g}$ be the Lie algebra of type $E_8$ with  $\mathfrak{n}_-\subset \widehat{\mathfrak{g}}$ a nilpotent subalgebra and $Q'\subset Q(G,\mathbb{C}^3)$ the subquiver supported on vertices $Q'_0=\{k+5,2k+3, 2k+4, 2k+5,3k+6, 4k+7, 5k+8, 6k+9, 7k+10\}$. There is an isomorphism of algebras $$U(\mathfrak{n}_-)\overset{\Theta}{\rightarrow}\mathcal{H}\langle \{\Psi(\chi_i^!)\}_{i\in Q'_0}\rangle.$$

\vspace{-0.1in}	
\begin{figure}[htbp!]
\[\begin{tikzcd}
	\mathsmaller{{2k+3}} & \mathsmaller{{2k+4}} & \mathsmaller{{2k+5}} & \mathsmaller{{3k+6}} & \mathsmaller{{4k+7}} & \mathsmaller{{5k+8}} & \mathsmaller{{6k+9}} & \mathsmaller{{7k+10}} \\
	&& \mathsmaller{{k+5}}
	\arrow[from=1-1, to=1-2]
	\arrow[from=2-3, to=1-3]
	\arrow[from=1-2, to=1-3]
	\arrow[from=1-4, to=1-3]
	\arrow[from=1-5, to=1-4]
	\arrow[from=1-6, to=1-5]
	\arrow[from=1-7, to=1-6]
	\arrow[from=1-8, to=1-7]
\end{tikzcd}\]
	\caption{$Q'=\widetilde{E}_{8}\subset Q(G,\mathbb{C}^3)$}
\end{figure}
	\end{enumerate} 
\label{MainThm}		
\end{thm}

\begin{proof}	
	We present an argument for $(1)$ with the verification of $(2)-(5)$ being completely analogous. 
	
	The McKay quiver $Q(G,\mathbb{C}^3)$ contains an $\widetilde{A}_{k}$ subquiver $Q'$ supported on the vertices $1,2,\hdots, k+1$. Notice that $Q'$ has no oriented cycles, hence, $U(\mathfrak{n}_-)$ is isomorphic to the composition subalgebra of $\mathcal{H}(Rep(Q'))$ (subalgebra generated by characteristic functions of simple representations), see Example $4.25$ in \cite{J}.  Then subsequent application of $(2)$, Theorem \ref{BKRmain} and Proposition \ref{GCohandQuiver} together with Theorem \ref{CohMin1} and Corollary \ref{HallAlgIso}  gives rise to the proposed isomorphism, see diagram on Figure \ref{diagram}  with embeddings, isomorphisms and correspondences below for a schematic summary.
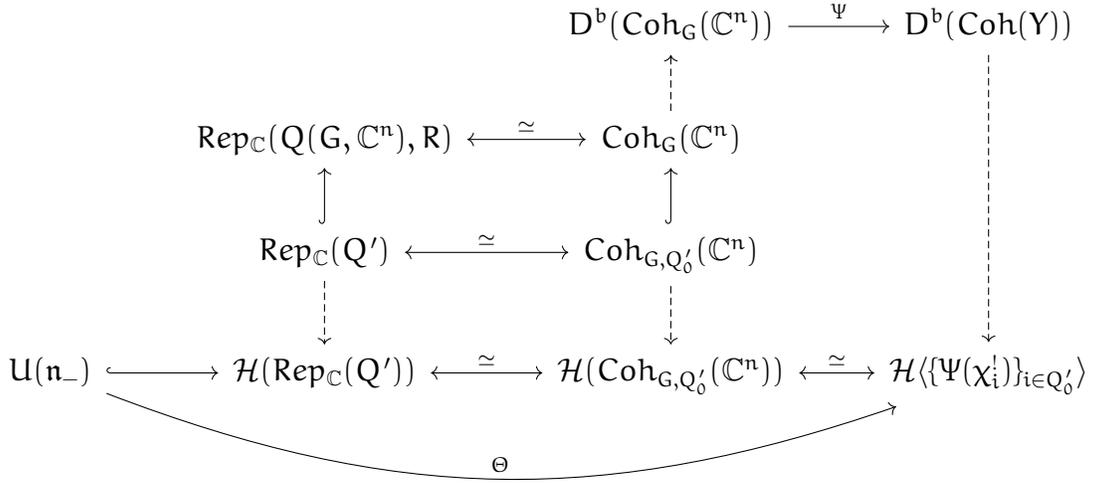
\begin{figure}[htbp!]	
\[\begin{tikzcd}
	&& {D^b(Coh_G(\mathbb{C}^n))} & {D^b(Coh(Y))} \\
	{} & {Rep_{\mathbb{C}}(Q(G,\mathbb{C}^n), R)} & {Coh_G(\mathbb{C}^n)} \\
	& {Rep_{\mathbb{C}}(Q')} & {Coh_{G,Q'_0}(\mathbb{C}^n)} \\
	{U(\mathfrak{n}_-)} & {\mathcal{H}(Rep_{\mathbb{C}}(Q'))} & {\mathcal{H}(Coh_{G,Q'_0}(\mathbb{C}^n))} & {\mathcal{H}\langle \{\Psi(\chi_i^!)\}_{i\in Q'_0}\rangle}
	\arrow["\simeq", <->, from=2-2, to=2-3]
	\arrow[hook, from=3-3, to=2-3]
	\arrow[hook, from=3-2, to=2-2]
	\arrow["\simeq", <->, from=3-2, to=3-3]
	\arrow[hook, from=4-1, to=4-2]
	\arrow["\simeq", <->, from=4-2, to=4-3]
	\arrow["\simeq", <->, from=4-3, to=4-4]
	\arrow[dashed, from=3-2, to=4-2]
	\arrow[dashed, from=3-3, to=4-3]
	\arrow[dashed, from=2-3, to=1-3]
	\arrow["\Psi", from=1-3, to=1-4]
	\arrow[dashed, from=1-4, to=4-4]
	\arrow["\Theta",bend right=20,  from=4-1, to=4-4]
\end{tikzcd}\]
\caption{Diagrammatic overview of the construction}
\label{diagram}
\end{figure}
\end{proof}
\begin{rmk}
Let $G=\mathbb{Z}/r\mathbb{Z}\hookrightarrow SL_3(\mathbb{C})$ be a finite abelian subgroup with $\nu_1=\dfrac{1}{r}(1,k,s)$ and $1+k+s =r$. The McKay quiver $Q(G,\mathbb{C}^3)$ can not contain a $\widetilde{D}_4$ subquiver. This is due to the fact that 	$\widetilde{D}_4$ has a vertex of valency $4$, which implies the existence of an oriented $3$-cycle, supported on these vertex and $2$ of the vertices connected to it. Indeed, let the vertex of valency $4$ correspond to the character (irreducible representation) $\chi_i$. Then, inevitably, there are two vertices indexed by $\chi_{i-a}$ and   $\chi_{i+b}$ with $a\neq b \in \{1,k,s\}$  that are included in the subgraph as well. Therefore, the subgraph contains an oriented $3$-cycle supported on the vertices enumerated by $\chi_i,\chi_{i-a}$ and   $\chi_{i+b}$.
\end{rmk}
 \begin{conj}
 	Let $ \Bbbk $ be a finite field and $\mbox{char}(\Bbbk)$ coprime with the order of $G$. Under the assumptions $(1)-(5)$ of Theorem \ref{MainThm} one has the corresponding isomorphisms 
 	
 	$$U_v(\mathfrak{n}_-)\overset{\Theta_{fin}}{\rightarrow}\mathcal{H}_{fin}\langle \{\Psi(\chi_i^!)\}_{i\in Q'_0}\rangle$$
 	
 	with  $v=\sqrt{\Bbbk}$.
 \end{conj}
\begin{rmk}
We provide some facts in support of the conjecture. The Bridgeland-King-Reid  equivalence ((2) in Theorem \ref{BKRmain}) is known to hold in the above setup (see the comment after Conjecture $2.24$ in \cite{Rou}). Another required modification to the proof of Theorem \ref{MainThm} is to use Theorem \ref{Ringel} in place of the isomorphism from Example $4.25$ in \cite{J}. The analogue of Corollary \ref{HallAlgIso} holds true as well.	
\end{rmk}

\bibliographystyle{alpha}

\begin{thebibliography}{99}
\bib{BK}{article}{
	author={Bezrukavnikov, R.},
	author={Kaledin, D.},
	title={McKay equivalence for symplectic resolutions of quotient
		singularities},
	language={Russian, with Russian summary},
	journal={Tr. Mat. Inst. Steklova},
	volume={246},
	date={2004},
	number={Algebr. Geom. Metody, Svyazi i Prilozh.},
	pages={20--42},
	issn={0371-9685},
	translation={
		journal={Proc. Steklov Inst. Math.},
		date={2004},
		number={3(246)},
		pages={13--33},
		issn={0081-5438},
	},
}

\bib{BKR}{article}{
	author={Bridgeland, T.},
	author={King, A.},
	author={Reid, M.},
	title={The McKay correspondence as an equivalence of derived categories},
	journal={J. Amer. Math. Soc.},
	volume={14},
	date={2001},
	number={3},
	pages={535--554},
}

\bib{CL}{article}{
	author={Cautis, S.},
	author={Logvinenko, T.},
	title={A derived approach to geometric McKay correspondence in dimension
		three},
	journal={J. Reine Angew. Math.},
	volume={636},
	date={2009},
	pages={193--236},
}	

\bib{CCL}{article}{
	author={Cautis, S.},
	author={Craw, A.},
	author={Logvinenko, T.},
	title={Derived Reid's recipe for abelian subgroups of ${\rm
			SL}_3(\Bbb{C})$},
	journal={J. Reine Angew. Math.},
	volume={727},
	date={2017},
	pages={1--48},
}

\bib{C1}{article}{
	author={Craw, A.},
	title={An explicit construction of the McKay correspondence for $A$-Hilb
		$\Bbb C^3$},
	journal={J. Algebra},
	volume={285},
	date={2005},
	number={2},
	pages={682--705},
}

\bib{C2}{article}{
	author = {Craw, A.},
	journal = {arXiv:0807.2191v1},
	title = {Quiver representations in toric geometry},
	date = {2008}}
\bib{CI}{article}{
	author={Craw, A.},
	author={Ishii, A.},
	title={Flops of $G$-Hilb and equivalences of derived categories by
		variation of GIT quotient},
	journal={Duke Math. J.},
	volume={124},
	date={2004},
	number={2},
	pages={259--307},
	issn={0012-7094},
}
\bib{CR}{article}{
	author={Craw, A},
	author={Reid, M.},
	title={How to calculate $A$-Hilb $\Bbb C^3$},
	conference={
		title={Geometry of toric varieties},
	},
	book={
		series={S\'{e}min. Congr.},
		volume={6},
		publisher={Soc. Math. France, Paris},
	},
	date={2002},
	pages={129--154},
}
\bib{DW}{book}{
	author={Derksen, H.},
	author={Weyman, J.},
	title={An introduction to quiver representations},
	series={Graduate Studies in Mathematics},
	volume={184},
	publisher={American Mathematical Society, Providence, RI},
	date={2017},
	pages={x+334},
}

\bib{IN}{article}{
	author={Ito, Y.},
	author={Nakajima, H.},
	title={McKay correspondence and Hilbert schemes in dimension three},
	journal={Topology},
	volume={39},
	date={2000},
	number={6},
	pages={1155--1191},
}

\bib{GSV}{article}{
	author={Gonzalez-Sprinberg, G.},
	author={Verdier, J.-L.},
	title={Construction g\'{e}om\'{e}trique de la correspondance de McKay},
	language={French},
	journal={Ann. Sci. \'{E}cole Norm. Sup. (4)},
	volume={16},
	date={1983},
	number={3},
	pages={409--449 (1984)},
}
\bib{J}{article}{
	author={Joyce, D.},
	title={Configurations in abelian categories. II. Ringel-Hall algebras},
	journal={Adv. Math.},
	volume={210},
	date={2007},
	number={2},
	pages={635--706},
}
\bib{KV}{article}{
	author={Kapranov, M.},
	author={Vasserot, E.},
	title={Kleinian singularities, derived categories and Hall algebras},
	journal={Math. Ann.},
	volume={316},
	date={2000},
	number={3},
	pages={565--576},
	issn={0025-5831}
}

\bib{KawT}{article}{
	author={Kawamata, Y.},
	title={Derived categories of toric varieties III},
	journal={Eur. J. Math.},
	volume={2},
	date={2016},
	number={1},
	pages={196--207}
}

\bib{McKay}{article}{
	author={McKay, J.},
	title={Graphs, singularities, and finite groups},
	conference={
		title={The Santa Cruz Conference on Finite Groups},
		address={Univ. California, Santa Cruz, Calif.},
		date={1979},
	},
	book={
		series={Proc. Sympos. Pure Math.},
		volume={37},
		publisher={Amer. Math. Soc., Providence, R.I.},
	},
	date={1980},
	pages={183--186},
}

\bib{Nak}{article}{
	author={Nakamura, I.},
	title={Hilbert schemes of abelian group orbits},
	journal={J. Algebraic Geom.},
	volume={10},
	date={2001},
	number={4},
	pages={757--779},
}

\bib{Re}{article}{
	author={Reid, M.},
	title={La correspondance de McKay},
	note={S\'{e}minaire Bourbaki, Vol. 1999/2000},
	journal={Ast\'{e}risque},
	number={276},
	date={2002},
	pages={53--72},
}

\bib{R}{article}{
	author={Ringel, C. M.},
	title={Hall algebras and quantum groups},
	journal={Invent. Math.},
	volume={101},
	date={1990},
	number={3},
	pages={583--591}
}	

\bib{Rou}{article}{
	author={Rouquier, Raphael},
	title={Derived equivalences and finite dimensional algebras},
	conference={
		title={International Congress of Mathematicians. Vol. II},
	},
	book={
		publisher={Eur. Math. Soc., Z\"{u}rich},
	},
	date={2006},
	pages={191--221},
}

\bib{S}{article}{
	author={Schiffmann, O.},
	title={Lectures on Hall algebras},
	language={English, with English and French summaries},
	conference={
		title={Geometric methods in representation theory. II},
	},
	book={
		series={S\'{e}min. Congr.},
		volume={24},
		publisher={Soc. Math. France, Paris},
	},
	date={2012},
	pages={1--141},
}




\end{thebibliography}

\end{document}